\documentclass[a4paper,10pt, oneside]{article}
\usepackage{amsmath,amssymb,amsthm,mathrsfs,graphicx}
\usepackage{authblk}
\usepackage[font=small,labelfont=md,textfont=it]{caption}
\usepackage{floatrow,float}
\usepackage[titletoc, title]{appendix}
\usepackage[colorlinks,linkcolor=blue,citecolor=blue]{hyperref}
\usepackage{etoolbox}
\usepackage{longtable}
\usepackage{diagbox}
\usepackage{booktabs,makecell,multirow}
\usepackage[capitalise,nosort]{cleveref}
\usepackage{cases,color}
\crefname{equation}{}{}
\crefname{lem}{Lemma}{Lemmas}
\crefname{thm}{Theorem}{Theorems}
\crefname{discr}{Discretization}{Discretizations}

\DeclareMathOperator{\D}{D}

\apptocmd{\sloppy}{\hbadness 10000\relax}{}{}

\newcommand{\dual}[1]{\langle {#1} \rangle}
\newcommand{\Dual}[1]{\left\langle {#1} \right\rangle}

\newcommand{\nm}[1]{\lVert {#1} \rVert}
\newcommand{\Nm}[1]{\left\lVert {#1} \right\rVert}
\newcommand{\snm}[1]{\lvert {#1} \rvert}

\newcommand{\Snm}[1]{\left\lvert {#1} \right\rvert}
\newcommand{\ssnm}[1]
{
  \left\vert\kern-0.25ex
  \left\vert\kern-0.25ex
  \left\vert
  {#1}
  \right\vert\kern-0.25ex
  \right\vert\kern-0.25ex
  \right\vert
}

\makeatletter
\def\spher@harm#1{%
  \vbox{\hbox{%
    \offinterlineskip
    \valign{&\hb@xt@2\p@{\hss$##$\hss}\vskip.2ex\cr#1\crcr}%
  }\vskip-.36ex}%
}
\def\gshone{\spher@harm{.}}
\def\gshtwo{\spher@harm{.&.}}
\def\gshthree{\spher@harm{.&.&.}}
\let\gsh\spher@harm
\makeatother

\newtheorem{lem}{Lemma}[section]
\newtheorem{rem}{Remark}[section]
\newtheorem{thm}{Theorem}[section]

\makeatletter\def\@captype{table}\makeatother

\begin{document}

\title{
  \Large \bf Numerical analysis of a semilinear fractional diffusion equation
  \thanks
  {
    This work was supported in part by the National Natural Science Foundation of
    China (NSFC) Grants Nos.~11901410 and 11771312.
  }
}
\author[1]{Binjie Li\thanks{libinjie@scu.edu.cn, libinjie@aliyun.com}}
\author[2]{Tao Wang\thanks{Corresponding author: wangtao5233@hotmail.com}}
\author[1]{Xiaoping Xie\thanks{xpxie@scu.edu.cn}}
\affil[1]{School of Mathematics, Sichuan University, Chengdu 610064, China}
\affil[2]{South China Research Center for Applied Mathematics and Interdisciplinary Studies, South China Normal University, Guangzhou 510631, China}
%\footnotetext[1]{School of Mathematics, Sichuan University}
%\footnotetext[2]{School of Mathematics, South China Normal University}

\renewcommand\Authands{ and }

\date{}
\maketitle

\begin{abstract}

  This paper considers the numerical analysis of a semilinear fractional diffusion
  equation with nonsmooth initial data. A new Gr\"onwall's inequality and its
  discrete version are proposed. By the two inequalities,   error estimates in
  three Sobolev norms are derived for a spatial semi-discretization and a full
  discretization, which are optimal with respect to the
  regularity of the solution. A sharp temporal error estimate on graded temporal
  grids is also rigorously established. In addition, the spatial accuracy $
  \scriptstyle O(h^2(t^{-\alpha} + \ln(1/h)\!)\!) $ in the pointwise $ \scriptstyle
  L^2(\Omega) $-norm is obtained for a spatial semi-discretization. Finally, several
  numerical results are provided to verify the theoretical results.

\end{abstract}


%The convergence analysis of the aforementioned algorithms is generally carried out on
%the condition that the underlying solution is sufficiently smooth. So far, the works
%on the numerical analysis for nonsmooth data are very limited. Using the Laplace
%transformation, McLean and Thom\'ee \cite{McLean2010IMA} analyzed three fully
%discretizations for fractional order evolution equations, where the initial values
%are allowed to have only $ L^2(\Omega) $-regularity. Using a growth estimation of the
%Mittag-Leffler function, Jin et al.~\cite{Jin2013SIAM,Jin2015IMA} analyzed the
%convergence of a spatial semi-discretization of problem \cref{eq:model}, and they
%derived the following error estimates: if $ f = 0 $, then
%\[
  %\nm{u(t) - u_h(t)}_{L^2(\Omega)} +
  %h \nm{u(t)-u_h(t)}_{H_0^1(\Omega)} \leqslant
  %C h^2 \snm{\ln h} t^{-\alpha} \nm{u_0}_{L^2(\Omega)};
%\]
%if $ u_0 = 0 $ and $ 0 \leqslant \beta < 1 $, then
%\[
  %%\label{eq:jin}
  %\nm{u\!-\!u_h}_{L^2(0,T;L^2(\Omega))} +
  %h \nm{u\!-\!u_h}_{L^2(0,T;H_0^1(\Omega))} \leqslant
  %C h^{2-\beta} \nm{f}_{L^2(0,T; \dot H^{-\beta}(\Omega))},
%\]
%\[
  %\nm{u(t)\!-\!u_h(t)}_{L^2(\Omega)}\!+\!
  %h \nm{u(t)\!-\!u_h(t)}_{H_0^1(\Omega)}\!\leqslant\!
  %C h^{2-\beta} \snm{\ln h}^2 \nm{f}_{L^\infty(0,t;\dot H^{-\beta}(\Omega))}.
%\]
%Recently, McLean and Mustapha \cite{McLean2015Time} derived that
%\[
  %\nm{u(t_n) - U^n}_{L^2(\Omega)} \leqslant
  %C t_n^{-1} \Delta t \nm{u_0}_{L^2(\Omega)},
%\]
%for a temporal semi-discretization of a fractional diffusion problem without source.
%For more related work, we refer the reader to \cite{Cuesta2006,McLean2010IMA}.

\section{Introduction}
Let $ 0 < T < \infty $ and $ \Omega \subset \mathbb R^d $ ($d=1,2,3$) be a convex $ d
$-polytope. We consider the  following semilinear fractional
diffusion equation:
\begin{equation}
  \label{eq:model}
  \D_{0+}^\alpha(u-u_0)(t) - \Delta u(t) = f(u(t)),
  \quad 0 < t \leqslant T,
\end{equation}
where $ 0 < \alpha < 1 $, $ u_0 \in L^2(\Omega) $, $ u(t) \in H_0^1(\Omega) $ for
a.e.~$ 0 < t \leqslant T $,  $ f: \mathbb R \to \mathbb R $ is Lipschitz
continuous, and $ \D_{0+}^\alpha $ is a Riemann-Liouville fractional
differential operator of order $ \alpha $.

%with weakly singular kernels, \cite{Mustapha2010IMA,Cuesta2006} a second-order
%accurate time discretization of a semilinear integro-differential equation with a
%weakly singular kernel was developed. This time discretization used graded grids to
%tackle the singularity of the solution in time; however, the numerical analysis was
%carried out with some regularity assumptions on the solution. For a semilinear time
%fractional wave equation, Cuesta et al.~\cite{Cuesta2006} proposed a convolution
%quadrature time discretization and presented a rigorous convergence analysis.

%In \cite{Mustapha2010IMA}, a second-order accurate time discretization of a
%semilinear integro-differential equation with a weakly singular kernel was analyzed;
%however, the numerical analysis was carried out with some regularity assumptions on
%the underlying solution.

%The field of numerical treatments for fractional diffusion-wave equations have
%attracted an extensive In the past thirty years, the field of numerical treatments
%for fractional diffusion-wave equations has attracted an extensive amount of effort.
By now there is an extensive literature on the numerical treatments of linear
fractional diffusion-wave equations. Roughly speaking, the  algorithms in the literature can be divided into   four types.
%We summarize part of the literature as follows. In \cite{Lubich1996,Jin2016} a
 The first type of algorithm   uses the convolution quadrature rules proposed in
\cite{Lubich1986,Lubich1988} to approximate the time fractional calculus operators; see, e.g.  \cite{Lubich1996,Jin2016}.
%For the algorithms of this class we refer the reader to \cite{Lubich1996,Jin2016}. A
The second type adopts the L1 scheme to approximate the time fractional
derivatives; see \cite{sun2006fully,Lin2007,Jin2015IMA,Li-Wang-Xie2019wave,Liao2019} and the
references therein.  The third type   employs the spectral method to approximate the
time fractional calculus operators; see
 \cite{Li2009,Zheng2015,Li2018-wave,Zayernouri2014Fractional,Zayernouri2012Karniadakis,Zayernouri2014Exponentially} %,yang2016spectral}
 and the references therein.  The fourth type of algorithm utilizes the discontinuous Galerkin or
Petrov-Galerkin method to approximate the time fractional calculus operators; see
\cite{Mustapha2009Discontinuous,Mclean2009Convergence,Mustapha2009Discontinuous,Mustapha2012Uniform,Mclean2015Time,Li2019SIAM,Luo2019}
and the references therein.

%A considerable amount of numerical algorithms for time fractional diffusion problems
%have been developed. Generally, these numerical algorithms can be divided into three
%types. This first type uses finite difference methods to approximate the time
%fractional derivatives. Despite their ease of implementation, the fractional
%difference methods are generally of temporal accuracy orders not greater than two;
%%see \cite{Yuste2005, Langlands2005, Yuste2006, Chen2007, Lin2007, Zhuang2008,
%%Deng2009, Zhang2009524, Chen2010,Cui2009, Liu2011,Gao2011,Zeng2013, Wang2014, Li2016}
%and the references therein. The second one employs the spectral method to discretize
%the time fractional derivatives;
%%see \cite{Li2009,
%%Zheng2015,Li2017spectral,Zayernouri2014Fractional,Zayernouri2012Karniadakis,Zayernouri2014Exponentially,yang2016spectral,Li2017A}.
%The main advantage of these algorithms is that they possess high-order accuracy,
%provided the solution is sufficiently smooth. The third type adopts the finite
%element method to approximate the time fractional derivatives; see
%%\cite{McLean2009Convergence, Mustapha2011Piecewise,
%%Mustapha2015Time,Li2017A,Mustapha2009Discontinuous,Mustapha2012Uniform,Mustapha2012Superconvergence,McLean2015Time,mustapha2014well-posedness,Mustapha2014A}.
%These algorithms are easy to implement as those in the first type, and they can
%possess high-order accuracy.

It is well known that solutions of time fractional diffusion-wave equations generally
have singularity in time, despite how smooth the initial data is.  %The influence of this fact
% is twofold. One is that
This means that rigorous numerical analyses for the time fractional diffusion-wave
equations are necessary. In \cite{sun2006fully} Sun and Wu proposed the well-known L1
scheme and derived temporal accuracies $ O(\tau^{2-\alpha}) $ and $ 
O(\tau^{3-\alpha}) $ for the fractional diffusion equations and the fractional wave
equations. However, Jin et al.~\cite{Jin2015IMA} proved that the L1 scheme is of only
temporal accuracy $ O(\tau) $ for the fractional diffusion equations with
nonvanishing initial data, and Li et al.~\cite{Li-Wang-Xie2019wave} proved that the
factor $ \tau^\alpha/h^2 $ will significantly worsen the temporal accuracy of the L1
scheme for the fractional wave equations with nonvanishing initial data.
%We note that graded temporal grids can be used to tackle the singularity in time \cite{Li-Luo-Xie2018spacetime}.

%We also note that Lin and Xu \cite{Lin2007} developed independently the L1 scheme for
%the fractional diffusion equation and obtained the temporal accuracy $ \mathcal
%O(\tau^{2-\alpha}) $. The analysis in the above two papers are carried out under the
%condition that the solution is sufficiently smooth.
%%Therefore, it is important to establish the
%%convergence of the L1 scheme without regularity restrictions on the solution.
%In this aspect,  Using the convolution quadratures in \cite{Lubich1988}
%and techniques in \cite{Lubich1996}, Jin et al.~\cite{Jin2016} developed two full
%discretizations for fractional diffusion-wave equations and established the
%convergence with nonsmooth data. Recently, Yan et al.~\cite{Yan2018} proposed a
%modified L1 scheme for the fractional diffusion equation, which possesses temporal
%accuracy $ \mathcal O(\tau^{2-\alpha}) $ for smooth and nonsmooth initial data in
%both homogeneous and inhomogeneous cases.

So far, the numerical analysis for time fractional diffusion-wave equations mainly
focuses on the linear problems, and is very rare for semilinear fractional diffusion
equations. For problem \cref{eq:model} with $ u_0 \in \dot H^2(\Omega) $, Jin et
al.~\cite{Jin2018} analyzed the regularity of the solution and proposed an elegant
numerical analysis framework for deriving the following pointwise $ L^2(\Omega)
$-norm error estimates:
\begin{equation}
  \label{eq:Jin-space}
  \max_{0 \leqslant t \leqslant T}
  \nm{u(t) - u_h(t)}_{L^2(\Omega)} \leqslant
  C (\ln(1/h))^2 h^2 \nm{u_0}_{\dot H^2(\Omega)}
\end{equation}
for a spatial
semi-discretization,
and
\[
  \max_{1 \leqslant n \leqslant N}
  \nm{u_h(t_n) - u_h^n}_{L^2(\Omega)} \leqslant C \tau^\alpha
  \nm{u_0}_{\dot H^2(\Omega)}
\]
 for a full discretization using the L1 scheme or the backward Euler CQ scheme in the
temporal discretization with   uniform temporal grids.
Applying the Newton linearized method to the full discretization with the L1 scheme
in \cite{Jin2018}, Li et al.~\cite{Li-Wu-Zhang2019} proposed a linearized Galerkin
finite element method. Under the condition that
\[
  \nm{u''(t)}_{L^\infty(\Omega)} \leqslant C(1+t^{\alpha-2})
\]
and $ u $ is sufficiently regular in spatial directions, they obtained the temporal
accuracy $ O(J^{-\alpha\sigma}) $ on graded temporal grids with $ 1 \leqslant \sigma
\leqslant \frac{2-\alpha}\alpha $ (cf.~\cref{sec:discretization} for the definitions
of $ J $ and $ \sigma $). We note that if $ f $ is Lipschitz continuous, then the
above regularity conditions on $ u $ do not hold necessarily. It should be mentioned
that there are two works on the numerical analysis for semilinear
integro-differential equations with weakly singular kernels. Cuesta et
al.~\cite{Cuesta2006} proposed a convolution quadrature time discretization with a
rigorous convergence analysis under weak assumptions on the data. This discretization
is of second order but requires the grids to be uniform. Mustapha and Mustapha
\cite{Mustapha2010IMA} developed a second-order-accurate time discretization on
graded grids with a convergence analysis under some regularity assumptions on the
solution.

%To our best knowledge, \cite{Jin2018} is the only work available that presents a
%rigorous convergence analysis for problem \cref{eq:model} with $ u_0 \in
%H_0^1(\Omega) \cap H^2(\Omega) $. We establish an energy type Gr\"onwall's
%inequality.

%The works in \cite{Jin2018,Li-Wu-Zhang2019} only consider the pointwise $ L^2(\Omega)
%$ norm error estimate, and only \cite{Jin2018} provides a complete numerical analysis
%for problem \cref{eq:model} with $ u_0 \in \dot H^2(\Omega) $. This motivates us to
%derive a numerical analysis for problem \cref{eq:model} with $ u_0 \in L^2(\Omega) $.
In this paper, we consider the numerical analysis of problem \cref{eq:model} with $
u_0 \in L^2(\Omega) $. Our main contributions  are as follows.
\begin{itemize}
  \item We establish the regularity of problem \cref{eq:model} with $ u_0 \in \dot
    H^{2\delta}(\Omega) $ for $ 0 \leqslant \delta \leqslant 1 $.
    \item  A new  energy type Gr\"onwall's inequality and its
    discrete version are proposed.
  %\item For a spatial semi-discretization, we obtain
    %\begin{equation}
      %\label{eq:Li-space}
      %\nm{(u-u_h)(t)}_{L^2(\Omega)} \leqslant C
      %h^2(t^{-\alpha} + \ln(1/h)) \nm{u_0}_{L^2(\Omega)}.
    %\end{equation}
    %The error estimates in the norms
    %\[
      %\nm{\cdot}_{{}_0H^{\alpha/2}(0,T;L^2(\Omega))}, \,
      %\nm{\cdot}_{L^2(0,T;\dot H^1(\Omega))} \text{ and }
      %\nm{\cdot}_{L^2(0,T;L^2(\Omega))}
    %\]
    %are also derived, and these error estimates are optimal with respect to the
    %regularity of the solution.
  \item For a spatial semi-discretization,
  %with piecewise linear conforming finte element method
 we obtain the error estimate
    \begin{equation*}
    %  \label{eq:Li-space}
      \nm{(u-u_h)(t)}_{L^2(\Omega)} \leqslant C
      h^2(t^{-\alpha} + \ln(1/h)) \nm{u_0}_{L^2(\Omega)}.
    \end{equation*}
%    which remove a $\ln(1/h)$ factor for the estimate in \cite[Theorem 4.4]{Jin2018}.
    \item  For a full discretization using a discontinuous Galerkin method in the
    temporal discretization, optimal error estimates with respect to the regularity of the solution are derived in the
    norms
    \[
      \nm{\cdot}_{{}_0H^{\alpha/2}(0,T;L^2(\Omega))}, \,
      \nm{\cdot}_{L^2(0,T;\dot H^1(\Omega))} \text{ and }
      \nm{\cdot}_{L^2(0,T;L^2(\Omega))}.
    \]
   % are  derived  in the case $ u_0 \in L^2(\Omega) $.
    \item    %using discontinuous Galerkin method in the
   % temporal discretization,
    A sharp temporal error
    estimate for the full discretization is established   on graded temporal grids in the case $ u_0 \in \dot H^1(\Omega) $.
    %\begin{equation}
      %\begin{aligned}
        %& \nm{u-u_h}_{L^2(0,T;L^2(\Omega))} +
        %h \nm{u-u_h}_{L^2(0,T;\dot H^1(\Omega))} \\
        %\leqslant{} &
        %C\nm{u_0}_{L^2(\Omega)}
        %\begin{cases}
          %h^2 & \text{ if }\, 0 < \alpha < 1/2, \\
          %\sqrt{\ln(1/h)} \, h^2 & \text{ if }\, \alpha = 1/2, \\
          %h^{1/\alpha} & \text{ if }\, 1/2 < \alpha < 1.
        %\end{cases}
      %\end{aligned}
    %\end{equation}
  %\item For a full discretization, in the case $ u_0 \in L^2(\Omega) $ we derive
    %optimal temporal error estimate with respect to the regularity of the solution,
    %and in the case $ u_0 \in \dot H^1(\Omega) $ our temporal error estimate reveals
    %how the graded temporal grid can improve the temporal accuracy.
%  \item In the case $ u_0 \in \dot H^1(\Omega) $ a sharp temporal error
%    estimate on graded temporal grids is also established.
\end{itemize}
%We note that our derivation of \cref{eq:Li-space} do not use the framework proposed
%in \cite{Jin2018}.
The energy type Gr\"onwall's inequality and its
discrete version are crucial to the error estimate in energy norms, and  it may be useful for the numerical analysis of corresponding optimal control problems. To our best knowledge, this
paper provides the first numerical analysis of problem \cref{eq:model} with nonsmooth
data.
%The error estimates in the $ {}_0H^{\alpha/2}(0,T;L^2(\Omega))
%$-norm, $ L^2(0,T;\dot H^1(\Omega)) $-norm and $ L^2(0,T;L^2(\Omega)) $-norm are
%derived by an energy type Gr\"onwall's inequality and its discrete version.

%In our analysis, an energy type Gr\"onwall's inequality and its discrete version play
%a crucial role. In this paper we derive an energy type Gr\"onwall's inequality. This
%inequality is used to prove the uniqueness of the weak solution to problem
%\cref{eq:model} and used to derive the error estimates of a spatial
%semi-discretization. Its discrete version is crucial for deriving the error estimates
%in the norms $ {}_0H^{\alpha/2}(0,T;L^2(\Omega)) $, $ L^2(0,T;\dot H^1(\Omega)) $ and
%$ L^2(0,T;L^2(\Omega)) $. We also note that our derivation of \cref{eq:Li-space} do
%not use the framework proposed in \cite{Jin2018}.

%In this paper we develop an energy type Gr\"onwall's inequality and use this
%inequality to establish the uniqueness of the weak solution and the convergence of a
%spatial semi-discretization in the $ {}_0H^{\alpha/2}(0,T;L^2(\Omega)) $-norm, $
%L^2(0,T;L^2(\Omega)) $-norm and $ L^2(0,T;\dot H^1(\Omega)) $-norm. We also use the
%discrete version of this energy type Gr\"onwall's inequality to derive the temporal
%error estimates of a full discretization.

The rest of this paper is organized as follows. \cref{sec:pre} introduces some
function spaces and the Riemann-Liouville fractional calculus operators.
\cref{sec:regu} investigates the regularity of problem \cref{eq:model} with nonsmooth
and smooth initial data. \cref{sec:semi} establishes the convergence of the spatial
semi-discretization with nonsmooth initial data. \cref{sec:discretization} derives
temporal error estimates for the full discretization. Finally, \cref{sec:numer}
provides several numerical experiments to verify the theoretical results.

\section{Preliminaries}
\label{sec:pre}
Assume that $ -\infty < a < b < \infty $ and $ X $ is a separable Hilbert space. For
each $ m \in \mathbb N $, define
\begin{align*}
  {}_0H^m(a,b;X) & := \{v\in H^m(a,b;X): v^{(k)}(a)=0,\,\, 0\leqslant k<m\}, \\
  {}^0H^m(a,b;X) & := \{v\in H^m(a,b;X): v^{(k)}(b)=0,\,\, 0\leqslant k<m\},
\end{align*}
where $ H^m(a,b;X) $ is a standard vector valued Sobolev space and $ v^{(k)} $ is the
$ k $-th weak derivative of $ v $. We equip the above two spaces with the norms
\begin{align*}
  \nm{v}_{{}_0H^m(a,b;X)} &:= \nm{v^{(m)}}_{L^2(a,b;X)}
  \quad \forall v \in {}_0H^m(a,b;X), \\
  \nm{v}_{{}^0H^m(a,b;X)} &:= \nm{v^{(m)}}_{L^2(a,b;X)}
  \quad \forall v \in {}^0H^m(a,b;X),
\end{align*}
respectively. For any $ m \in \mathbb N_{>0} $ and $ 0 < \theta < 1 $, define
\begin{align*}
  {}_0H^{m-1+\theta}(a,b;X) &:= [{}_0H^{m-1}(a,b;X), \ {}_0H^m(a,b;X)]_{\theta,2}, \\
  {}^0H^{m-1+\theta}(a,b;X) &:= [{}^0H^{m-1}(a,b;X), \ {}^0H^m(a,b;X)]_{\theta,2},
\end{align*}
where $ [\cdot, \cdot]_{\theta,2} $ means the famous $ K $-method \cite[Chapter
22]{Tartar2007}. For each $ 0 < \gamma < \infty $, we use $ {}^0H^{-\gamma}(a,b;X) $
to denote the dual space of $ {}_0H^\gamma(a,b;X) $ and use the notation $
\dual{\cdot,\cdot}_{{}_0H^\gamma(a,b;X)} $ to denote the duality paring between $
{}^0H^{-\gamma}(a,b;X) $ and $ {}_0H^\gamma(a,b;X) $. The space $
{}_0H^{-\gamma}(a,b;X) $ and the notation $ \dual{\cdot,\cdot}_{{}^0H^\gamma(a,b;X)}
$ are defined analogously.

For $ -\infty < \gamma < 0 $, define
\begin{align*}
  \left(\D_{a+}^\gamma v\right)(t) &:=
  \frac1{ \Gamma(-\gamma) }
  \int_a^t (t-s)^{-\gamma-1} v(s) \, \mathrm{d}s, \quad a < t < b, \\
  \left(\D_{b-}^\gamma v\right)(t) &:=
  \frac1{ \Gamma(-\gamma) }
  \int_t^b (s-t)^{-\gamma-1} v(s) \, \mathrm{d}s, \quad a < t < b,
\end{align*}
for all $ v \in L^1(a,b;X) $, where $ \Gamma(\cdot) $ is the gamma function. In
addition, let $ \D_{a+}^0 $ and $ \D_{b-}^0 $ be the identity operator on $
L^1(a,b;X) $. For $ j - 1 < \gamma \leqslant j $ with $ j \in \mathbb N_{>0} $,
define
\begin{align*}
  \D_{a+}^\gamma v & := \D^j \D_{a+}^{\gamma-j}v, \\
  \D_{b-}^\gamma v & := (-\D)^j \D_{b-}^{\gamma-j}v,
\end{align*}
for all $ v \in L^1(a,b;X) $, where $ \D $ is the first-order differential operator
in the distribution sense.

\begin{lem}
	\label{lem:regu-basic}
  Assume that $ 0 < \gamma < \infty $. If $ v \in {}_0H^\gamma(a,b;X) $, then
  \begin{small}
  \begin{align*}
    C_1 \nm{v}_{{}_0H^\gamma(a,b;X)} \leqslant
    \nm{\D_{a+}^\gamma v}_{L^2(a,b;X )} \leqslant
    C_2 \nm{v}_{{}_0H^\gamma(a,b;X)},
  \end{align*}
  \end{small}
  where $ C_1 $ and $ C_2 $ are two positive constants depending only on $ \gamma $.
  If $ v \in L^2(a,b;X) $, then
  \[
    \nm{\D_{a+}^{-\gamma}v}_{{}_0H^\gamma(a,b;X)}
    \leqslant C \nm{v}_{L^2(a,b;X)},
  \]
  where $ C $ is a positive constant depending only on $ \gamma $. In particular,
  \[
    \nm{\D_{a+}^{-\gamma} v}_{{}_0H^\gamma(a,b;X)}
    \leqslant \frac{C}{\sqrt\gamma} \nm{v}_{L^2(a,b;X)}
  \]
  for all $ v \in L^2(a,b;X) $ and $ 0 < \gamma \leqslant \beta < 1 $, where $ C $ is
  a positive constant depending only on $ \beta $.
\end{lem}
\begin{lem}
  \label{lem:coer}
  If $ 0 < \gamma < 1/2 $, then
  \begin{align*}
    \cos(\gamma\pi) \nm{\D_{a+}^\gamma v}_{L^2(a,b;X)}^2 \leqslant
    (\D_{a+}^\gamma v, \D_{b-}^\gamma v)_{L^2(a,b;X)} \leqslant
    \sec(\gamma\pi) \nm{\D_{a+}^\gamma v}_{L^2(a,b;X)}^2, \\
    \cos(\gamma\pi) \nm{\D_{b-}^\gamma v}_{L^2(a,b;X)}^2 \leqslant
    (\D_{a+}^\gamma v, \D_{b-}^\gamma v)_{L^2(a,b;X)} \leqslant
    \sec(\gamma\pi) \nm{\D_{b-}^\gamma v}_{L^2(a,b;X)}^2,
  \end{align*}
  for all $ v \in {}_0H^\gamma(a,b;X) $ (equivalent to $ {}^0H^\gamma(a,b;X) $),
  where $ (\cdot,\cdot)_{L^2(a,b;X)} $ is the usual inner product in $ L^2(a,b;X) $.
  Moreover,
  \[
    \dual{\D_{a+}^{2\gamma} v, w}_{{}^0H^\gamma(a,b;X)} =
    (\D_{a+}^\gamma v, \D_{b-}^\gamma w)_{L^2(a,b;X)} =
    \dual{\D_{b-}^{2\gamma} w, v}_{{}_0H^\gamma(a,b;X)}
  \]
  for all $ v \in {}_0H^\gamma(a,b;X) $ and $ w \in {}^0H^\gamma(a,b;X) $.
\end{lem}
\begin{rem}
  For the proofs of the above two lemmas, we refer the reader to \cite{Ervin2006} and
  \cite[Section 3]{Luo2019}.
\end{rem}

It is well known that there exists an orthonormal basis $\{\phi_n: n \in \mathbb N \}
\subset H_0^1(\Omega) \cap H^2(\Omega) $ of $ L^2(\Omega) $ such that
\[
	-\Delta \phi_n =\lambda_n \phi_n,
\]
where $ \{ \lambda_n: n \in \mathbb N \} $ is a positive non-decreasing sequence and
$\lambda_n\to\infty$ as $n\to\infty$.
For any $ -\infty< \beta < \infty $, define
\begin{center}
  $
  \dot H^\beta(\Omega) := \Big\{
    \sum_{n=0}^\infty v_n \phi_n:\
    \sum_{n=0}^\infty \lambda_n^\beta v_n^2 < \infty
  \Big\}
  $,
\end{center}
and endow this space with the norm
\[
  %\Nm{\sum_{n=0}^\infty v_n \phi_n}_{\dot H^\beta(\Omega)}
  \Big\|\sum_{n=0}^\infty v_n \phi_n  \Big\|_{\dot H^\beta(\Omega)}
  := \Big(
    \sum_{n=0}^\infty \lambda_n^\beta v_n^2
  \Big)^{1/2}.
\]

Finally, we introduce the following conventions: if $ D \subset \mathbb
R^l(l=1,2.3,4) $ is Lebesgue measurable, then $ \dual{p,q}_D := \int_D p \cdot q $
for scalar  or vector valued functions $ p $ and $ q $; the notation $ C_\times $
means a positive constant depending only on its subscript(s), and its value may
differ at each occurrence; let $ L $ be the Lipschitz constant of $ f $ and assume $
f(0) = 0 $.

\section{Regularity}
\label{sec:regu}
Define a bounded linear operator
\[
  S: {}_0H^{-\alpha/2}(0,T;L^2(\Omega)) \to {}_0H^{\alpha/2}(0,T;L^2(\Omega))
  \cap L^2(0,T;\dot H^1(\Omega))
\]
by that
\begin{equation}
  \label{eq:weak-form-S}
  \dual{\D_{0+}^\alpha Sg, v}_{{}^0H^{\alpha/2}(0,T;L^2(\Omega))} +
  \dual{\nabla Sg, \nabla v}_{\Omega \times (0,T)} =
  \dual{g, v}_{{}^0H^{\alpha/2}(0,T;L^2(\Omega))}
\end{equation}
for all $ g \in {}_0H^{-\alpha/2}(0,T;L^2(\Omega)) $ and $ v \in
{}^0H^{\alpha/2}(0,T;L^2(\Omega)) \cap L^2(0,T;\dot H^1(\Omega)) $. For any $ g \in
{}_0H^\beta(0,T;L^2(\Omega)) $ with $ -\alpha/2 \leqslant \beta < \infty $, it holds
\begin{equation}
  \label{eq:regu-S}
  \nm{Sg}_{{}_0H^{\alpha+\beta}(0,T;L^2(\Omega))} +
  \nm{Sg}_{{}_0H^\beta(0,T;\dot H^2(\Omega))}
  \leqslant C_{\alpha,\beta} \nm{g}_{{}_0H^\beta(0,T;L^2(\Omega))}.
\end{equation}
We call $ u \in {}_0H^{\alpha/2}(0,T;L^2(\Omega)) \cap L^2(0,T;\dot H^1(\Omega)) $ a
weak solution to problem \cref{eq:model} if
\begin{equation}
  \label{eq:weak_sol}
  u = S \D_{0+}^\alpha u_0 + S f(u).
\end{equation}
\begin{rem}
  For the well-posedness of $ S $ and the derivation of \cref{eq:regu-S}, we refer
  the reader to \cite{Li2009,Li2019SIAM,Luo2019}.
\end{rem}
\begin{rem}
  Assume that $ u $ is a weak solution to problem \cref{eq:model}. By
  \cref{eq:regu-S} we have
  \begin{small}
  \[
    \nm{u-S\D_{0+}^\alpha u_0}_{{}_0H^\alpha(0,T;L^2(\Omega))} +
    \nm{u-S\D_{0+}^\alpha u_0}_{L^2(0,T;\dot H^2(\Omega))}
    \leqslant C_{\alpha,L} \nm{u}_{L^2(0,T;L^2(\Omega))},
  \]
  \end{small}
  so that from \cref{eq:weak-form-S,lem:regu-basic,lem:regu-u0-growth} it follows
  that
  \begin{equation}
    \label{eq:weak_sol_int}
    \dual{\D_{0+}^\alpha(u-u_0), v}_{\Omega \times (0,T)} +
    \dual{\nabla u, \nabla v}_{\Omega \times (0,T)} =
    \dual{f(u), v}_{\Omega \times (0,T)}
  \end{equation}
  for all $ v \in {}^0H^{\alpha/2}(0,T;L^2(\Omega)) \cap L^2(0,T;\dot H^1(\Omega)) $.
\end{rem}

The main results of this section are the following two theorems.
\begin{thm}
  \label{thm:regu}
  Problem \cref{eq:model} admits a unique weak solution $ u $ such that
  \begin{equation}
    \label{eq:regu-C}
    \nm{u}_{C([0,T];L^2(\Omega))} \leqslant
    C_{\alpha,L,T} \nm{u_0}_{L^2(\Omega)}.
  \end{equation}
  %Moreover, for any $ 0 < \epsilon < 1/4 $,
  %\begin{small}
  %\begin{equation}
  %\begin{aligned}
  %& \nm{
  %\D_{0+}^{\alpha+1/2-\epsilon}(u-S\D_{0+}^\alpha u_0)
  %}_{L^2(0,T;L^2(\Omega))} +
  %\nm{
  %\D_{0+}^{1/2-\epsilon}(u-S\D_{0+}^\alpha u_0)
  %}_{L^2(0,T;\dot H^2(\Omega))} \\
  %\leqslant{} &
  %C_{\alpha,L,T,\Omega} \epsilon^{-1/2} \nm{u_0}_{L^2(\Omega)}.
  %\end{aligned}
  %\end{equation}
  %\end{small}
  %\begin{equation}
  %\label{eq:regu}
  %\begin{aligned}
  %& \nm{u-S\D_{0+}^\alpha u_0}_{{}_0H^{\alpha+1/2-\epsilon}(0,T;L^2(\Omega))} +
  %\nm{u-S\D_{0+}^\alpha u_0}_{{}_0H^{1/2-\epsilon}(0,T;\dot H^2(\Omega))} \\
  %\leqslant{} &
  %C_{\alpha,L,T} \epsilon^{-1/2} \nm{u_0}_{L^2(\Omega)}.
  %\end{aligned}
  %\end{equation}
  If $ 0 < \alpha < 1/3 $, then
  \begin{equation}
    \nm{u}_{{}_0H^{\alpha/2}(0,T;\dot H^2(\Omega))} \leqslant
    C_{\alpha,L,T,\Omega} \nm{u_0}_{L^2(\Omega)}.
  \end{equation}
  If $ \alpha=1/3 $, then, for any $ 0 < \epsilon < 1 $,
  \begin{equation}
    \nm{u}_{{}_0H^{\alpha/2}(0,T;\dot H^{2-\epsilon}(\Omega))} \leqslant
    C_{\alpha,L,T,\Omega} \epsilon^{-1/2} \nm{u_0}_{L^2(\Omega)}.
  \end{equation}
  If $ 1/3 < \alpha < 1 $, then
  \begin{equation}
    \label{eq:73}
    \nm{u}_{{}_0H^{\alpha/2}(0,T;\dot H^{1/\alpha-1}(\Omega))} \leqslant
    C_{\alpha,L,T,\Omega} \nm{u_0}_{L^2(\Omega)}.
  \end{equation}
  If $ 0 < \alpha < 1/2 $, then
  \begin{equation}
    \label{eq:regu-12}
    \nm{u}_{L^2(0,T;\dot H^2(\Omega))}
    \leqslant C_{\alpha,L,T,\Omega} \nm{u_0}_{L^2(\Omega)}.
  \end{equation}
  If $ \alpha=1/2 $, then, for any $ 0 < \epsilon < 1 $,
  \begin{equation}
    \label{eq:regu-13}
    \nm{u}_{L^2(0,T;\dot H^{2-\epsilon}(\Omega))} \leqslant
    C_{\alpha,L,T,\Omega} \epsilon^{-1/2} \nm{u_0}_{L^2(\Omega)}.
  \end{equation}
  If $ 1/2 < \alpha <1 $, then
  \begin{equation}
    \label{eq:regu-14}
    \nm{u}_{L^2(0,T;\dot H^{1/\alpha}(\Omega))} \leqslant
    C_{\alpha,L,T,\Omega} \nm{u_0}_{L^2(\Omega)}.
  \end{equation}
  Moreover, for any $ 0 < \epsilon < 1/4 $,
  \begin{equation}
    \label{eq:regu-11}
    \nm{u}_{{}_0H^{1/2-\epsilon}(0,T;L^2(\Omega))} \leqslant
    C_{\alpha,L,T} \epsilon^{-1/2} \nm{u_0}_{L^2(\Omega)}.
  \end{equation}
\end{thm}
% Note: $ 0 < \epsilon <1/4 $ is used to ensure that the constant $ C_{\alpha,L,T} $
% in \cref{eq:regu} does not depend on $ \epsilon $.
\begin{thm}
  \label{thm:regu-higher}
  Assume that $ u_0 \in \dot H^{2\delta}(\Omega) $ with $ 0 < \delta \leqslant 1 $.
  Then the weak solution $ u $ to problem \cref{eq:model} satisfies that $ u' \in
  C((0,T];L^2(\Omega)) $ and
  \begin{equation}
    \label{eq:lsj-1}
    \sup_{0 < t \leqslant T} t^{1-\alpha \delta}
    \nm{u'(t)}_{L^2(\Omega)} \leqslant
    C_{\alpha,L,T,\Omega} \delta^{-1} \nm{u_0}_{\dot H^{2\delta}(\Omega)}.
  \end{equation}
  Moreover, if $ 0 < \delta \leqslant 1/2 $, then % Note: why $ 1/2 $?
  \begin{small}
  \begin{equation}
    \label{eq:lsj-2}
    \sup_{0 < t \leqslant T}
    t^{1-\alpha(\delta+1/2)} \nm{u'(t) - (S\D_{0+}^\alpha u_0)'(t)}_{\dot H^1(\Omega)}
    \leqslant C_{\alpha,L,T,\Omega} \delta^{-2}
    \nm{u_0}_{\dot H^{2\delta}(\Omega)}.
  \end{equation}
  \end{small}
  In particular, if $ \delta=1/2 $, then
  \begin{equation}
    \label{eq:u'-h1}
    \sup_{0 < t \leqslant T } t \nm{u'(t)}_{\dot H^1(\Omega)}
    \leqslant C_{\alpha,L,T,\Omega} \nm{u_0}_{\dot H^1(\Omega)}.
  \end{equation}
\end{thm}
\begin{rem}
  By \cref{eq:weak_sol}, \cref{eq:regu-C}, \cref{lem:911} and
  \cref{lem:regu-u0-growth}, it is evident that $ u(0) = u_0 $.
\end{rem}
\begin{rem} % checked
  For $ u_0 \in \dot H^2(\Omega) $, \cite[Theorem 3.1]{Jin2018} has already contained
  the following regularity estimate:
  \[
    \nm{u'(t)}_{L^2(\Omega)} \leqslant C t^{\alpha-1}
    \nm{u_0}_{\dot H^2(\Omega)}, \quad 0 < t \leqslant T.
  \]
%  However, the proof therein is not complete since it was only proved as an a priori
%  estimate.
  We also note that a simple modification of the proof of \cref{eq:lsj-2}
  gives
  \begin{small}
  \[
    \sup_{0 < t \leqslant T } t^{1-\alpha(\delta+1)}
    \nm{u'(t) - (S\D_{0+}^\alpha u_0)'(t) - E(t)f(u_0)}_{L^2(\Omega)}
    \leqslant  C_{\alpha,L,T,\Omega} \delta^{-2}
    \nm{u_0}_{\dot H^{2\delta}(\Omega)},
  \]
  \end{small}
  for all $ u_0 \in \dot H^{2\delta}(\Omega) $ with $ 0 < \delta \leqslant 1 $, where
  $ E $ is defined by \cref{eq:E-def}.
\end{rem}
\begin{rem}
  For a semilinear fractional evolution equation with an almost sectorial operator,
  Wang et al.~\cite{Wang-Chen-Xiao2019} has derived the unique existence of the mild
  solution and the classical solution.
\end{rem}

The purpose of the remaining of this section is to prove the above two theorems. To
this end, let us first introduce an integral representation of $ S $. For any $ g \in
L^p(0,T;L^2(\Omega)) $ with $ p > 1/\alpha $, we have that \cite{Jin2018}
\begin{equation}
  \label{eq:S-E}
  Sg(t) = \int_0^t E(t-s) g(s) \, \mathrm{d}s,
  \quad 0 \leqslant t \leqslant T,
\end{equation}
where
\begin{equation}
  \label{eq:E-def}
  E(t) := \frac1{2\pi i} \int_0^\infty e^{-rt}\big(
    (r^\alpha e^{-i\alpha\pi} - \Delta)^{-1} -
    (r^\alpha e^{i\alpha\pi} - \Delta)^{-1}
  \big) \, \mathrm{d}r.
\end{equation}
%The above $ \Delta $ is understood as a linear operator defined on $ \dot H^1(\Omega)
%$.
% Todo: Maybe it is better to briefly explain \cref{eq:S-E}.

\begin{lem} % checked
  \label{lem:E}
  The $ \mathcal L(L^2(\Omega),\!\dot H^1\!(\Omega)) $-valued function $ E $ is
  analytic on $ (0,\infty) $ with
  \begin{equation}
    \label{eq:E}
    \nm{E(t)}_{\mathcal L(L^2(\Omega))} +
    t^{\alpha/2} \nm{E(t)}_{\mathcal L(L^2(\Omega),\dot H^1(\Omega))} +
    t\nm{E'(t)}_{\mathcal L(L^2(\Omega))} \leqslant
    C_\alpha t^{\alpha-1}
  \end{equation}
  for all $ t > 0 $. For any $ 0 \leqslant \delta \leqslant 1/2 $,
  \begin{equation}
    \label{eq:2delta-h1}
    \sup_{t > 0} t^{1-\alpha(1/2+\delta)}
    \nm{E(t)}_{\mathcal L(\dot H^{2\delta}(\Omega), \dot H^1(\Omega))}
    \leqslant C_\alpha.
  \end{equation}
  In addition, for any $ 0 < t < T $ and $ 0 < \Delta t \leqslant
  T-t $,
  \begin{equation}
    \label{eq:E2}
    \int_0^t \Nm{ E(s+\Delta t) - E(s) }_{
      \mathcal L(L^2(\Omega))
    } \, \mathrm{d}s \leqslant C_{\alpha,T,\Omega} (\Delta t)^\alpha.
  \end{equation}
\end{lem}
\begin{proof}
  It is evident that $ E $ is an analytic $ \mathcal L(L^2(\Omega), \dot H^1(\Omega))
  $-valued function on $ (0,\infty) $, and we refer the reader to \cite{Jin2018} for
  the proof of \cref{eq:E,eq:2delta-h1}. Now let us prove \cref{eq:E2}. By
  \cref{eq:E-def} and the fact that
  \begin{small}
  \[
    \nm{
      (r^\alpha e^{\pm i\alpha \pi} - \Delta)^{-1}
    }_{\mathcal L(L^2(\Omega))} \leqslant
    C_{\alpha,\Omega} (1+r^\alpha)^{-1}
    \text{ for all } r \geqslant 0,
  \]
  \end{small}
  we obtain
  \begin{small}
  \begin{align*}
    & \int_0^t \nm{ E(s+\Delta t) - E(s) }_{
      \mathcal L(L^2(\Omega))
    } \, \mathrm{d}s \\
    \leqslant{} & C_{\alpha,\Omega}
    \int_0^t \int_0^\infty (1-e^{-r\Delta t}) e^{-rs}
    (1+r^\alpha)^{-1} \, \mathrm{d}r \, \mathrm{d}s \\
    \leqslant{}& C_{\alpha,\Omega} (\mathbb I_1 + \mathbb I_2),
  \end{align*}
  \end{small}
  where
  \begin{small}
  \begin{align*}
    \mathbb I_1 &:= t \int_0^1 1-e^{-r\Delta t} \, \mathrm{d}r, \\
    \mathbb I_2 &:= \int_0^t \int_1^\infty (1-e^{-r\Delta t})
    e^{-rs} r^{-\alpha} \, \mathrm{d}r \, \mathrm{d}s.
  \end{align*}
  \end{small}

  A simple calculation gives that
  \begin{small}
  \begin{align*}
    \mathbb I_1 & = t(\Delta t)^{-1}(\Delta t + e^{-\Delta t} - 1) <
    t (\Delta t)^{-1} \big( 1 - (1-\Delta t)e^{\Delta t} \big) \\
    & < t (\Delta t)^{-1}
    \big( 1 - (1-\Delta t)(1+\Delta t) \big) = t \Delta t,
  \end{align*}
  \end{small}
  and
  \begin{small}
  \begin{align*}
    \mathbb I_2 &=
    \int_1^\infty (1-e^{-r\Delta t})
    r^{-1-\alpha}(1-e^{-rt}) \, \mathrm{d}r <
    \int_1^\infty (1-e^{-r\Delta t}) r^{-1-\alpha} \, \mathrm{d}r \\
    & =
    (\Delta t)^\alpha \int_{\Delta t}^\infty
    s^{-1-\alpha}(1-e^{-s }) \, \mathrm{d}s <
    C_\alpha (\Delta t)^\alpha.
  \end{align*}
  \end{small}
  Finally, combining the above two estimates proves \cref{eq:E2} and hence this
  lemma.
\end{proof}

\begin{lem} % checked
  \label{lem:911}
  If $ v \in C([0,T];L^2(\Omega)) $, then
  \[
    Sf(v) \in C([0,T];L^2(\Omega)) \text{ with }
    (Sf(v))(0) = 0
  \]
  and
  \begin{align}
    %\label{eq:911-1}
    \nm{Sf(v)}_{C([0,T];L^2(\Omega))} \leqslant
    C_{\alpha,L,T} \nm{v}_{C([0,T];L^2(\Omega))}.
  \end{align}
  Moreover, if $ v, w \in C([0,T];L^2(\Omega)) $, then
  \begin{equation}
    \label{eq:Sfv-Sfu}
    \nm{Sf(v) - Sf(w)}_{C([0,T];L^2(\Omega))} \leqslant
    C_{\alpha,L,T} \nm{v-w}_{C([0,T];L^2(\Omega))}.
  \end{equation}
\end{lem}
\begin{lem} % checked
  \label{lem:g}
  Let
  \begin{equation}
    \label{eq:g-def}
    g(t) := (S\D_{0+}^\alpha u_0)'(t) + E(t) f(u_0),
    \quad 0 < t \leqslant T,
  \end{equation}
  and assume that $ u_0 \in \dot H^{2\delta}(\Omega) $ with $ 0 \leqslant \delta
  \leqslant 1 $.  Then $ g \in C((0,T];L^2(\Omega)) $ and
  \begin{equation}
    \label{eq:g}
    \sup_{0 < t \leqslant T}
    t^{1-\alpha\delta} \nm{g(t)}_{L^2(\Omega)}
    \leqslant C_{\alpha,L,T}
    \nm{u_0}_{\dot H^{2\delta}(\Omega)}.
  \end{equation}
  Moreover, if $ 1/2 \leqslant \delta \leqslant 1 $ then $ g \in C((0,T];\dot
  H^1(\Omega)) $ and
  \begin{equation}
    \sup_{0 < t \leqslant T}
    t^{1-\alpha(\delta-1/2)} \nm{g(t)}_{\dot H^1(\Omega)}
    \leqslant C_{\alpha,L,T} \nm{u_0}_{\dot H^{2\delta}(\Omega)}.
  \end{equation}
\end{lem}
\begin{lem}
  \label{lem:w}
  Let $ g $ be defined by \cref{eq:g-def}, assume that $ u_0 \in \dot
  H^{2\delta}(\Omega) $ with $ 0 < \delta \leqslant 1 $, and define
  \[
    w(t) := (S\D_{0+}^\alpha u_0)(t) +
    \int_0^t E(t-s) f(v(s)) \, \mathrm{d}s,
    \quad 0 \leqslant t \leqslant T.
  \]
  If $ v \in C([0,T];L^2(\Omega)) \cap C^1((0,T];L^2(\Omega)) $ satisfies
  \[
    \sup_{0 < t \leqslant T} t^{1-\alpha\delta}
    \nm{v'(t)}_{L^2(\Omega)} < \infty,
  \]
  then $ w \in C([0,T];L^2(\Omega)) \cap C^1((0,T];L^2(\Omega)) $,
  \[
    w'(t) = g(t) + \int_0^t E(t - s) f'(v(s)) v'(s) \, \mathrm{d}s,
    \quad 0 < t < T,
  \]
  and
  \begin{equation}
    \sup_{0 < t \leqslant T} t^{1-\alpha\delta}
    \nm{w'(t)}_{L^2(\Omega)} < \infty.
  \end{equation}
\end{lem}

\begin{rem}
  By \cref{eq:S-E,lem:E,lem:regu-u0-growth}, a straightforward calculation proves
  \cref{lem:911,lem:g}. Then, by \cref{lem:E,lem:g,lem:regu-u0-growth}, a routine
  calculation gives \cref{lem:w} (we refer the reader to the proof of \cite[Theorem
  3.1]{Jin2018} for the relevant techniques).
\end{rem}

Then let us present a Gr\"onwall-type inequality, which, together with its discrete
version, is crucial in this paper.
\begin{lem} % checked
  \label{lem:gronwall}
  Assume that $ 0 < \beta < 1/2 $, and $ A $ and $ \epsilon $ are two positive
  constants. If $ y \in {}_0H^\beta(0,T) $ satisfies that
  \begin{equation}
    \label{eq:grondwall-cond}
    \nm{\D_{0+}^\beta y}_{L^2(0,t)}^2 \leqslant
    \epsilon + A \nm{y}_{L^2(0,t)}^2
    \quad \text{ for all } 0 < t < T,
  \end{equation}
  then
  \begin{equation}
    \label{eq:gronwall}
    \nm{\D_{0+}^\beta y}_{L^2(0,t)} \leqslant C_{\beta,A,T} \, \sqrt\epsilon
    \quad \text{ for all } 0 < t < T.
  \end{equation}
\end{lem}
\begin{proof}
  For any $ 0 < t < T $, a straightforward computation yields
  \begin{align*}
    & \nm{y}_{L^2(0,t)}^2 = \dual{y,y}_{(0,t)} =
    \dual{y, \D_{t-}^\beta \D_{t-}^{-\beta} y}_{(0,t)}  \\
    ={} &
    \dual{\D_{0+}^\beta y, \D_{t-}^{-\beta} y}_{(0,t)}
    \leqslant \nm{\D_{0+}^\beta y}_{L^2(0,t)}
    \nm{\D_{t-}^{-\beta} y}_{L^2(0,t)}  \\
    \leqslant{} & C_\beta \nm{\D_{0+}^\beta y}_{L^2(0,t)}
    \nm{\D_{0+}^{-\beta} y}_{L^2(0,t)}
    \quad\text{(by \cref{lem:coer})}  \\
    \leqslant{} &
    \frac1{2A} \nm{\D_{0+}^\beta y}_{L^2(0,t)}^2 +
    C_{\beta,A} \nm{\D_{0+}^{-\beta} y}_{L^2(0,t)}^2,
  \end{align*}
  so that \cref{eq:grondwall-cond} implies
  \[
    \nm{\D_{0+}^{2\beta} \D_{0+}^{-\beta} y}_{L^2(0,t)}^2
    \leqslant 2\epsilon + C_{\beta,A} \nm{\D_{0+}^{-\beta} y}_{L^2(0,t)}^2.
  \]
  Letting $ k $ be the maximum integer satisfying that $ 2^k \beta \leqslant 1/2 $
  and repeating the above argument $ k-1 $ times, we obtain
  \[
    \nm{\D_{0+}^{2^k\beta} g_0}_{L^2(0,t)}^2
    \leqslant 2^k\epsilon + C_{\beta,A}
    \nm{g_0}_{L^2(0,t)}^2, \quad 0 < t< T,
  \]
  where $ g_0 := \D_{0+}^{-(2^k-1)\beta} y $. By the same techniques, we have
  \begin{equation}
    \label{eq:732}
    \nm{\D_{0+}^{2^k\beta + 1/4} g}_{L^2(0,t)}^2 \leqslant
    2^{k+1} \epsilon + C_{\beta,A} \nm{g}_{L^2(0,t)}^2,
    \quad 0 < t < T,
  \end{equation}
  where $ g := \D_{0+}^{-1/4} g_0 $. Since
  \begin{align*}
    \snm{ g(t) } &= \Snm{
      \big( \D_{0+}^{-2^k\beta-1/4} \D_{0+}^{2^k\beta+1/4} g \big)(t)
    } \\
    &= \Snm{
      \frac1{\Gamma(2^k\beta+1/4)} \int_0^t (t-s)^{2^k\beta-3/4}
      \D_{0+}^{2^k\beta+1/4} g(s) \, \mathrm{d}s
    } \\
    & \leqslant
    \frac1{\Gamma(2^k\beta+1/4)} \sqrt{
      \int_0^t (t-s)^{2^{k+1}\beta - 3/2} \, \mathrm{d}s
    } \, \nm{\D_{0+}^{2^k\beta + 1/4} g}_{L^2(0,t)} \\
    & \leqslant
    C_{\beta,T} \nm{\D_{0+}^{2^k\beta + 1/4} g}_{L^2(0,t)},
  \end{align*}
  it follows that
  \begin{align*}
    \snm{g(t)}^2 \leqslant C_{\beta,A,T}
    \Big( \epsilon + \nm{g}_{L^2(0,t)}^2 \Big),
    \quad 0 < t < T.
  \end{align*}
  %Hence,
  %letting $ \eta(t) := \nm{g}_{L^2(0,t)}^2 $ yields
  %\[
  %\eta'(t) \leqslant C_{\beta,A,T} \big( \epsilon + \eta(t) \big),
  %\quad 0 < t < T,
  %\]
  %and then applying the standard Gr\"onwall's inequality yields
  Applying the well-known Gr\"onwall's inequality then yields
  \[
   \nm{g}_{L^2(0,t)}^2 \leqslant C_{\beta,A,T} \epsilon,
    \quad 0 < t < T.
  \]
  Therefore, \cref{eq:gronwall} follows from \cref{eq:732} and the equality
  \[
    \D_{0+}^\beta y = \D_{0+}^{2^k\beta+1/4} g.
  \]
  This completes the proof.
\end{proof}

Finally, we are in a position to prove \cref{thm:regu,thm:regu-higher} as follows.

\medskip\noindent{\bf Proof of \cref{thm:regu}.} Let us first prove that problem
  \cref{eq:model} admits a weak solution which satisfies \cref{eq:regu-C}. For each
  $ k \in \mathbb N_{>0} $, define
  \begin{equation}
    \label{eq:vk}
    v_k := S\D_{0+}^\alpha u_0 + Sf(v_{k-1}),
  \end{equation}
  where $ v_0 := S\D_{0+}^\alpha u_0 $. By \cref{lem:911,lem:regu-u0-growth} we have
  that
  \[
    v_k \in C([0,T];L^2(\Omega)) \quad \text{ for each } k \in \mathbb N
  \]
  and
  \begin{equation}
    \label{eq:v0_v1}
    \nm{v_0}_{C([0,T];L^2(\Omega))} +
    \nm{v_1-v_0}_{C([0,T];L^2(\Omega))} \leqslant
    C_{\alpha,L,T} \nm{u_0}_{L^2(\Omega)}.
  \end{equation}
  By \cref{eq:S-E,eq:E}, a routine calculation gives that, for any $ k \in \mathbb
  N_{>0} $ and $ 0 < t \leqslant T $,
  \[
    \nm{(v_{k+1} - v_k)(t)}_{L^2(\Omega)} \leqslant
    \frac{(L C_\alpha t^\alpha \Gamma(\alpha))^k}{\Gamma(k\alpha+1)}
    \nm{v_1-v_0}_{C([0,T];L^2(\Omega))}.
  \]
  Since
  \[
    \sum_{k=1}^\infty \frac{(L C_\alpha T^\alpha \Gamma(\alpha))^k}
    {\Gamma(k\alpha+1)} < \infty,
  \]
  %\[
  %\lim_{k \to \infty} \frac{\Gamma(k\alpha+1)}{\Gamma(k\alpha+\alpha+1)} = 0,
  %\]
  it follows that
  \begin{equation}
    \label{eq:93}
    \sum_{k=0}^\infty \nm{v_{k+1} - v_k}_{C([0,T];L^2(\Omega))}
    \leqslant C_{\alpha,L,T} \nm{v_1-v_0}_{C([0,T];L^2(\Omega))}.
  \end{equation}
  Hence, $ \{v_k\}_{k=0}^\infty $ is a Cauchy sequence in $ C([0,T];L^2(\Omega)) $.
  Letting $ u $ be the limit of this Cauchy sequence, by \cref{eq:Sfv-Sfu} we have
  \[
    Sf(u) = \lim_{k \to \infty} Sf(v_k)
    \quad \text{ in } C([0,T];L^2(\Omega)),
  \]
  so that passing to the limit $ k \to \infty $ in \cref{eq:vk} yields
  \[
    u = S\D_{0+}^\alpha u_0 + Sf(u).
  \]
  Therefore, $ u $ is a weak solution to problem \cref{eq:model}. Moreover, using
  \cref{eq:v0_v1,eq:93} gives
  \begin{align*}
    \nm{u}_{C([0,T];L^2(\Omega))} =
    \lim_{k \to \infty} \nm{v_k}_{C([0,T];L^2(\Omega))}
    \leqslant C_{\alpha,L,T} \nm{u_0}_{L^2(\Omega)},
  \end{align*}
  which proves \cref{eq:regu-C}.

  Next, let us prove that the weak solution $ u $ is unique. To this end, assume that
  $ \widetilde u \neq u $ is another weak solution to problem \cref{eq:model}.
  Letting $ e = u - \widetilde u $, by \cref{eq:weak_sol_int} we have that, for any $
  0 < t \leqslant T $,
  \begin{align*}
    \dual{\D_{0+}^\alpha e, e}_{\Omega \times (0,t)} +
    \nm{e}_{L^2(0,t;\dot H^1(\Omega))}^2 =
    \dual{f(u) - f(\widetilde u), e}_{\Omega \times (0,t)}.
  \end{align*}
  From the inequality
  \begin{align*}
    \dual{f(u) - f(\widetilde u), e}_{\Omega \times (0,t)}
    & \leqslant L\nm{e}_{L^2(0,t;L^2(\Omega))}^2,
  \end{align*}
  it follows that
  \[
    \dual{\D_{0+}^\alpha e, e}_{\Omega \times (0,t)}
    \leqslant L\nm{e}_{L^2(0,t;L^2(\Omega))}^2,
  \]
  so that \cref{lem:coer} implies
  \begin{equation}
    \label{eq:e}
    \nm{\D_{0+}^{\alpha/2} e}_{L^2(0,t;L^2(\Omega))}
    \leqslant C_{\alpha,L} \nm{e}_{L^2(0,t;L^2(\Omega))}.
  \end{equation}
  Using \cref{lem:gronwall} then yields $ e = 0 $, which contradicts the assumption
  that $ u \neq \widetilde u $. This proves that the weak solution $ u $ to problem
  \cref{eq:model} is unique.

  %Finally, let us prove \cref{eq:regu} for the case $ 1/2 \leqslant \alpha < 1 $. By
  %\cref{eq:regu-C} we have
  %\[
  %\nm{f(u)}_{L^2(0,T;L^2(\Omega))} \leqslant C_{\alpha,L,T}
  %\nm{u_0}_{L^2(\Omega)},
  %\]
  %so that using \cref{eq:regu-S,eq:weak_sol} gives
  %\[
  %\nm{u-S\D_{0+}^\alpha u_0}_{{}_0H^\alpha(0,T;L^2(\Omega))} =
  %\nm{Sf(u)}_{{}_0H^\alpha(0,T;L^2(\Omega))}
  %\leqslant C_{\alpha,L,T} \nm{u_0}_{L^2(\Omega)}.
  %\]
  %From \cref{lem:regu-u0} it follows that
  %\[
  %\nm{u}_{{}_0H^{1/2-\epsilon}(0,T;L^2(\Omega))} \leqslant
  %C_{\alpha,L,T} \epsilon^{-1/2} \nm{u_0}_{L^2(\Omega)}.
  %\]
  %Hence, applying \cite[Lemma 28.1]{Tartar2007} gives
  %\[
  %\nm{f(u)}_{{}_0H^{1/2-\epsilon}(0,T;L^2(\Omega))} \leqslant
  %C_{\alpha,L,T} \epsilon^{-1/2} \nm{u_0}_{L^2(\Omega)},
  %\]
  %and then using \cref{eq:regu-S} again yields \cref{eq:regu}. Since the case $ 0 <
  %\alpha < 1/2 $ can be proved analogously, this completes the proof.
  %Finally, let us prove \cref{eq:regu} for the case $ 1/2 \leqslant \alpha < 1 $.

  Now let us prove \cref{eq:73,eq:regu-14,eq:regu-11} under the condition that $ 1/2
  < \alpha < 1 $. By \cref{eq:regu-C} we have
  \[
    \nm{f(u)}_{L^2(0,T;L^2(\Omega))} \leqslant C_{\alpha,L,T}
    \nm{u_0}_{L^2(\Omega)},
  \]
  so that using \cref{eq:regu-S,eq:weak_sol} gives
  \[
    \nm{u-S\D_{0+}^\alpha u_0}_{{}_0H^\alpha(0,T;L^2(\Omega))} =
    \nm{Sf(u)}_{{}_0H^\alpha(0,T;L^2(\Omega))}
    \leqslant C_{\alpha,L,T} \nm{u_0}_{L^2(\Omega)}.
  \]
  From \cref{eq:regu-u0-11} it follows that
  \[
    \nm{u}_{{}_0H^{\alpha/2}(0,T;L^2(\Omega))} \leqslant
    C_{\alpha,L,T,\Omega} \nm{u_0}_{L^2(\Omega)}.
  \]
  Hence, applying \cite[Lemma 28.1]{Tartar2007} gives
  \[
    \nm{f(u)}_{{}_0H^{\alpha/2}(0,T;L^2(\Omega))} \leqslant
    C_{\alpha,L,T,\Omega} \nm{u_0}_{L^2(\Omega)},
  \]
  and then using \cref{eq:regu-S,eq:weak_sol} again yields
  \begin{align*}
    & \nm{u-S\D_{0+}^\alpha u_0}_{{}_0H^{3\alpha/2}(0,T;L^2(\Omega))} +
    \nm{u-S\D_{0+}^\alpha u_0}_{{}_0H^{\alpha/2}(0,T;\dot H^2(\Omega))} \\
    ={} &
    \nm{Sf(u)}_{{}_0H^{3\alpha/2}(0,T;L^2(\Omega))} +
    \nm{Sf(u)}_{{}_0H^{\alpha/2}(0,T;\dot H^2(\Omega))} \\
    \leqslant{} &
    C_{\alpha,L,T,\Omega} \nm{u_0}_{L^2(\Omega)}.
  \end{align*}
  Therefore, \cref{eq:73,eq:regu-14,eq:regu-11} follow from
  \cref{eq:regu-u0-3,eq:regu-u0-6,eq:regu-u0-11}, respectively.

  %Finally, since \cref{eq:regu-11} in the case $ 0 < \alpha \leqslant 1/2 $,
  %\cref{eq:regu-12,eq:regu-13} can be proved analogously, this completes the proof.
  Finally, since the rest of this theorem can be proved analogously, this completes
  the proof.
\hfill\ensuremath{\blacksquare}

\medskip\noindent{\bf Proof of \cref{thm:regu-higher}.} Firstly, by
  \cref{lem:g,lem:w,lem:regu-u0-growth}, it is easily verified that the sequence $
  \{v_k\}_{k=1}^\infty $ defined in the proof of \cref{thm:regu} has the following
  property: for each $ k \in \mathbb N_{>0} $, $ v_k \in C^1((0,T];L^2(\Omega)) $ and
  \begin{equation}
    \label{eq:v_k'}
    v_k'(t) = g(t) +
    \int_0^t E(t-s) f'(v_{k-1}(s)) v_{k-1}'(s) \, \mathrm{d}s,
    \quad 0 < t \leqslant T,
  \end{equation}
  where $ g(t) $ is defined by \cref{eq:g-def}. A routine calculation then gives, by
  \cref{lem:E}, that, for any $ k \in \mathbb N_{>0} $ and $ 0 < t \leqslant T $,
  \begin{equation}
    \label{eq:426}
    t^{1-\alpha\delta} \nm{v_k'(t)}_{L^2(\Omega)} \leqslant
    A \Gamma(\alpha\delta)
    \sum_{j=0}^k \frac{(LC_\alpha t^\alpha \Gamma(\alpha))^j}
    {\Gamma(\alpha(j+\delta))},
  \end{equation}
  %\begin{equation}
  %\label{eq:427}
  %t^{1-\alpha\delta + \alpha/2}
  %\nm{v_k'(t)}_{\dot H^1(\Omega)} \leqslant
  %A + A \Gamma(\alpha\delta) \frac{\Gamma(\alpha/2)}{\Gamma(\alpha)}
  %\sum_{j=1}^k \frac{
  %(LC_\alpha t^\alpha \Gamma(\alpha))^j
  %}{\Gamma(\alpha(j+\delta-1/2))},
  %\end{equation}
  where
  \[
    A := \sup_{0 < t \leqslant T} t^{1-\alpha \delta}
    \nm{g(t)}_{L^2(\Omega)}.
  \]
  Since
  \[
    \sum_{j=0}^\infty \frac{
      (LC_\alpha T^\alpha \Gamma(\alpha))^j
    }{\Gamma(\alpha(j+\delta))}
    \leqslant C_{\alpha,L,T}
    \text{ and }
    \Gamma(\alpha\delta) \leqslant C_\alpha \delta^{-1},
  \]
  it follows that
  \[
    \sup_{0 < t \leqslant T} t^{1-\alpha\delta}
    \nm{v_k'(t)}_{L^2(\Omega)} \leqslant
    C_{\alpha,L,T} \delta^{-1} A \quad \forall k \in \mathbb N_{>0}.
  \]
  Therefore, \cref{lem:g} implies
  \begin{equation}
    \label{eq:v_k'-growth}
    \sup_{0 < t \leqslant T} t^{1-\alpha\delta}
    \nm{v_k'(t)}_{L^2(\Omega)}
    \leqslant C_{\alpha,L,T,\Omega} \delta^{-1}
    \nm{u_0}_{\dot H^{2\delta}(\Omega)} \quad \forall k \in \mathbb N_{>0}.
  \end{equation}

  Secondly, let us prove \cref{eq:lsj-1}. For any $ k \in \mathbb N_{>0} $ and $ a
  \leqslant t_1 < t_2 \leqslant T $, by \cref{eq:v_k'} we have
  \begin{align*}
    \big( v_k'-g \big)(t_2) - \big( v_k' - g \big)(t_1) =  \mathbb I_1 +
    \mathbb I_2 + \mathbb I_3,
  \end{align*}
  where
  \begin{align*}
    \mathbb I_1 &:=  \int_0^{a/2} \big( E(t_2-s) - E(t_1-s) \big)
    f'(v_{k-1}(s)) v_{k-1}'(s) \, \mathrm{d}s, \\
    \mathbb I_2 &:=  \int_{a/2}^{t_1} \big( E(t_2-s) - E(t_1-s) \big)
    f'(v_{k-1}(s)) v_{k-1}'(s) \, \mathrm{d}s, \\
    \mathbb I_3 &:= \int_{t_1}^{t_2} E(t_2-s)
    f'(v_{k-1}(s)) v_{k-1}'(s) \, \mathrm{d}s.
  \end{align*}
  By \cref{lem:E,eq:v_k'-growth}, a straightforward computation gives
  \begin{align*}
    \nm{\mathbb I_1}_{L^2(\Omega)} & \leqslant
    C_{a,\alpha,\delta,L,T,\Omega} (t_2-t_1) \nm{u_0}_{\dot H^{2\delta}(\Omega)}, \\
    \nm{\mathbb I_2}_{L^2(\Omega)} +
    \nm{\mathbb I_3}_{L^2(\Omega)} & \leqslant
    C_{a,\alpha,\delta,L,T,\Omega} (t_2-t_1)^\alpha
    \nm{u_0}_{\dot H^{2\delta}(\Omega)},
  \end{align*}
  so that $ \{v_k'-g\}_{k=1}^\infty $ is equicontinuous in $ C([a,T];L^2(\Omega)) $.
  In addition, \cref{eq:v_k'-growth,lem:g} imply that $ \{v_k'-g\}_{k=1}^\infty $ is
  uniformly bounded in $ C([a,T];L^2(\Omega)) $, and the proof of \cref{thm:regu}
  implies that
  \[
    \lim_{k \to \infty} v_k = u \quad\text{ in }
    C([0,T];L^2(\Omega)).
  \]
  Therefore, the celebrated Arzel\'a-Ascoli theorem yields
  \begin{equation}
    \label{eq:v_k-to-u-C1}
    \lim_{k \to \infty} v_k = u \quad \text{ in } C^1([a,T];L^2(\Omega)),
  \end{equation}
  so that passing to the limit $ k \to \infty $ in \cref{eq:v_k'-growth} gives
  \[
    \sup_{a \leqslant t \leqslant T} t^{1-\alpha\delta}
    \nm{u'(t)}_{L^2(\Omega)} \leqslant C_{\alpha,L,T,\Omega} \delta^{-1}
    \nm{u_0}_{\dot H^{2\delta}(\Omega)}.
  \]
  Since $ 0 < a < T $ is arbitrary, this proves \cref{eq:lsj-1}.

  Thirdly, let us prove \cref{eq:lsj-2}. By \cref{eq:weak_sol}, \cref{eq:lsj-1},
  \cref{eq:S-E} and \cref{lem:E}, a straightforward computation yields
  \[
    u'(t) = g(t) + \int_0^t E(t-s) f'(u(s)) u'(s) \, \mathrm{d}s,
    \quad 0 < t \leqslant T.
  \]
  From \cref{eq:lsj-1,eq:E} it follows that
  \begin{align*}
    \nm{u'(t) - g(t)}_{\dot H^1(\Omega)} \leqslant
    C_{\alpha,L,T,\Omega} \delta^{-2} t^{\alpha(1/2+\delta) - 1}
    \nm{u_0}_{\dot H^{2\delta}(\Omega)},
    \quad 0 < t \leqslant T.
  \end{align*}
  In addition, \cref{eq:2delta-h1} implies that
  \begin{align*}
    \nm{E(t)f(u_0)}_{\dot H^1(\Omega)} & \leqslant
    C_\alpha t^{\alpha(1/2+\delta)-1} \nm{f(u_0)}_{\dot H^{2\delta}(\Omega)} \\
    & \leqslant
    C_{\alpha,L} t^{\alpha(1/2+\delta)-1}
    \nm{u_0}_{\dot H^{2\delta}(\Omega)}.
  \end{align*}
  Therefore, \cref{eq:lsj-2} follows from the above two estimates.

  Finally, combining \cref{eq:lsj-2,eq:u0-h1-growth} proves \cref{eq:u'-h1} and thus
  concludes the proof of this theorem.
\hfill\ensuremath{\blacksquare}

\section{Spatial semi-discretization}
\label{sec:semi}
Let $ \mathcal K_h $ be a conventional conforming and shape regular triangulation of
$ \Omega $ consisting of $ d $-simplexes, and we use $ h $ to denote the maximum
diameter of these elements in $ \mathcal K_h $. Define
\[
  \mathcal V_h := \{
    v_h \in H_0^1(\Omega):\ v_h|_K \text{ is linear for each } K \in \mathcal K_h
  \}.
\]
Let $ P_h $ be the $ L^2(\Omega) $-orthogonal projection operator onto $ \mathcal V_h
$, and define the discrete Laplace operator $ \Delta_h: \mathcal V_h \to \mathcal V_h
$ by that
\[
  \dual{-\Delta_h v_h, w_h}_\Omega = \dual{\nabla v_h, \nabla w_h}
\]
for all $ v_h, w_h \in \mathcal V_h $. We use $ L_h^2(\Omega) $ to denote the space $
\mathcal V_h $ endowed with the norm of $ L^2(\Omega) $, and, for any $ \beta > 0 $,
we use $ \dot H_h^\beta(\Omega) $ to denote the space $ \mathcal V_h $ endowed with
the norm of $ \dot H^\beta(\Omega) $.

This section considers the following spatial semi-discretization:
\begin{equation}
  \label{eq:u_h}
  \D_{0+}^\alpha(u_h - P_h u_0)(t) - \Delta_h u_h(t) = P_h f(u_h(t)),
  \quad 0 < t \leqslant T.
\end{equation}
Let $ S_h $ be the spatially discrete version of $ S $, and we call $ u_h \in
{}_0H^{\alpha/2}(0,T;L_h^2(\Omega)) $ a weak solution to problem \cref{eq:u_h} if
\[
  u_h = S_h \D_{0+}^\alpha P_h u_0 + S_h P_h f(u_h).
\]
Following the proof of \cref{thm:regu,thm:regu-higher}, we can easily prove that
problem \cref{eq:u_h} possesses a unique weak solution $ u_h $ and
\begin{equation}
  \label{eq:uh'-h1}
  \sup_{0 < t \leqslant T} \big(
    t \nm{u'_h}_{\dot H^1(\Omega)} +
    t^{1-\alpha/2} \nm{u'_h}_{L^2(\Omega)}
  \big) \leqslant
  C_{\alpha,L,T,\Omega} \nm{P_h u_0}_{\dot H^1(\Omega)}.
\end{equation}
%\begin{rem}
%Similar to \cref{eq:regu,eq:u'-h1}, we have that
%\begin{small}
%\begin{equation}
%\label{eq:regu-uh-1}
%\begin{aligned}
%& \nm{u_h - S_h\D_{0+}^\alpha P_h u_0}_{
%{}_0H^{\alpha+1/2-\epsilon}(0,T;L_h^2(\Omega))
%} + \nm{
%u_h - S_h\D_{0+}^\alpha P_h u_0
%}_{{}_0H^{1/2-\epsilon}(0,T;\dot H_h^2(\Omega))} \\
%& \leqslant C_{\alpha,L,T} \epsilon^{-1/2} \nm{u_0}_{L^2(\Omega)},
%\end{aligned}
%\end{equation}
%\end{small}
%for all $ 0 < \epsilon < 1/4 $, and
%\begin{equation}
%\label{eq:regu-uh-2}
%\sup_{0 < t \leqslant T}
%t \nm{u_h'(t)}_{\dot H_h^1(\Omega)} \leqslant
%C_{\alpha,L,T,\Omega} \nm{P_hu_0}_{\dot H_h^1(\Omega)}.
%\end{equation}
%We have
%\end{rem}

\begin{thm}
  \label{thm:space}
  For each $ 0 < t \leqslant T $,
  \begin{equation}
    \label{eq:space}
    \nm{(u-u_h)(t)}_{L^2(\Omega)} \lesssim
    h^2(t^{-\alpha} + \ln(1/h)) \nm{u_0}_{L^2(\Omega)}.
  \end{equation}
  Moreover,
  \begin{small}
  \begin{equation}
    \label{eq:u-u_h-frac}
    \nm{u-u_h}_{{}_0H^{\alpha/2}(0,T;L^2(\Omega))} \lesssim
    \nm{u_0}_{L^2(\Omega)}
    \begin{cases}
      h^2 & \text{ if } 0 < \alpha < 1/3, \\
      \sqrt{\ln(1/h)} h^2 & \text{ if } \alpha = 1/3, \\
      h^{1/\alpha-1} & \text{ if } 1/3 < \alpha < 1,
    \end{cases}
  \end{equation}
  \end{small}
  and
  \begin{equation}
    \label{eq:u-u_h-energy}
    \begin{aligned}
      & \nm{u-u_h}_{L^2(0,T;L^2(\Omega))} +
      h \nm{u-u_h}_{L^2(0,T;\dot H^1(\Omega))} \\
      \lesssim{} &
      \nm{u_0}_{L^2(\Omega)}
      \begin{cases}
        h^2 & \text{ if }\, 0 < \alpha < 1/2, \\
        \sqrt{\ln(1/h)} \, h^2 & \text{ if }\, \alpha = 1/2, \\
        h^{1/\alpha} & \text{ if }\, 1/2 < \alpha < 1.
      \end{cases}
    \end{aligned}
  \end{equation}
\end{thm}
\noindent Above and throughout, $ h $ is assumed to be less than $ 1/2 $, and $ a
\lesssim b $ means that there exists a positive constant, depending only on $ \alpha
$, $ L $, $ T $, $ \Omega $ or the shape regularity of $ \mathcal K_h $, unless
otherwise specified, such that $ a \leqslant Cb $.

\begin{rem}
  Under the condition that $ u_0 \in \dot H^2(\Omega) $, a simple modification of the
  proof of \cref{eq:space} yields
  \[
    \sup_{0 \leqslant t \leqslant T}
    \nm{(u-u_h)(t)}_{L^2(\Omega)} \lesssim
  h^2  \ln(1/h)  \nm{u_0}_{\dot H^2(\Omega)}.
  \]
  We also note that Jin et~al.~\cite{Jin2018} derived that
  \[
    \sup_{0 \leqslant t \leqslant T}
    \nm{(u-u_h)(t)}_{L^2(\Omega)} \lesssim
    h^2  (\ln(1/h))^2 \nm{u_0}_{\dot H^2(\Omega)}.
  \]
  %In addition, our proof of \cref{eq:space} do not use the framework in
  %\cite{Jin2018}.
  %In addition, compared with the proof in \cite{Jin2018}, our proof uses a
  %Gr\"onwall-type inequality (cf.~\cref{lem:gronwall-2}) and thus avoids using the
  %maximal $ L^p $-regularity theory of fractional diffusion equations.
\end{rem}

The rest of this section is devoted to proving the above theorem. Similar to
\cref{eq:S-E},
%for any $ g_h \in
%{}_0H^\beta(0,T;L_h^2(\Omega)) $ with $ -\alpha/2 \leqslant \beta < \infty $, we have
%\[
  %(\D_{0+}^\alpha - \Delta_h) S_h g_h = g_h
%\]
%and
%%\begin{small}
%\begin{equation}
  %%\label{eq:regu-S}
  %\nm{S_hg}_{{}_0H^{\alpha+\beta}(0,T;L_h^2(\Omega))} +
  %\nm{S_hg}_{{}_0H^\beta(0,T;\dot H_h^2(\Omega))}
  %\leqslant C_{\alpha,\beta} \nm{g_h}_{{}_0H_h^\beta(0,T;L_h^2(\Omega))};
%\end{equation}
%%\end{small}
for any $ g_h \in L^p(0,T;L_h^2(\Omega)) $ with $ p > 1/\alpha $, we have
\begin{equation}
  \label{eq:S_h-E_h}
  S_hg_h(t) = \int_0^t E_h(t-s) g_h(s) \, \mathrm{d}s,
  \quad 0 \leqslant t \leqslant T,
\end{equation}
where
\begin{equation}
  \label{eq:E_h-def}
  E_h(t) := \frac1{2\pi i} \int_0^\infty e^{-rt}\big(
    (r^\alpha e^{-i\alpha\pi} - \Delta_h)^{-1} -
    (r^\alpha e^{i\alpha\pi} - \Delta_h)^{-1}
  \big) \, \mathrm{d}r.
\end{equation}
A trivial modification of the proof of \cref{eq:E} yields
\begin{equation}
  \label{eq:E_h}
  \nm{E_h(t)}_{\mathcal L(L^2(\Omega))}
  \leqslant C_\alpha t^{\alpha-1},
  \quad t > 0.
\end{equation}

%We have
%\begin{equation}
  %u_h = S_h(\D_{0+}^\alpha P_h u_0 + P_h f(u_h)).
%\end{equation}

\begin{lem}
  \label{lem:Lubich}
  If $ v \in L^2(\Omega) $, then
  \begin{equation}
    \label{eq:err_u0}
    \nm{
      \big( S\D_{0+}^\alpha v - S_h\D_{0+}^\alpha P_h v \big)(t)
    }_{L^2(\Omega)} \lesssim h^2 t^{-\alpha} \nm{v}_{L^2(\Omega)}
  \end{equation}
  for all $ 0 < t \leqslant T $. If $ g \in L^\infty(0,T;L^2(\Omega)) $, then
  \begin{equation}
    \label{eq:err_g}
    \nm{(Sg - S_h P_h g)(t)}_{L^2(\Omega)} \lesssim
    h^2 \ln(1/h) \nm{g}_{L^\infty(0,t;L^2(\Omega))}
  \end{equation}
  for all $ 0 < t \leqslant T $.
\end{lem}
\begin{proof}
  The proof of \cref{eq:err_u0} is similar to that of \cite[Theorem 2.1]{Lubich1996}
  and so is omitted. For any $ r \geqslant 0 $ and $ \theta = \pm \pi $, the proof of
  \cite[Theorem 2.1]{Lubich1996} contains that
	\[
		\nm{
			(r^\alpha e^{i\alpha\theta} -\Delta)^{-1} -
			(r^\alpha e^{i\alpha\theta} -\Delta_h)^{-1}P_h
		}_{\mathcal L(L^2(\Omega))}
		\lesssim h^2,
	\]
	and it is evident that
	\[
		\nm{
			(r^\alpha e^{i\alpha\theta} -\Delta)^{-1} -
			(r^\alpha e^{i\alpha\theta} -\Delta_h)^{-1}P_h
		}_{\mathcal L(L^2(\Omega))}
		\lesssim (1+r^\alpha)^{-1}.
	\]
	Consequently, by \cref{eq:E-def,eq:E_h-def}, a straightforward computation gives
	\[
		\nm{E(s) - E_h(s)P_h}_{\mathcal L(L^2(\Omega))}
		\lesssim \min\{s^{-1} h^2, 1+s^{\alpha-1}\}
	\]
	for all $ s > 0 $. Therefore, by \cref{eq:E,eq:S_h-E_h} we obtain
	\[
		\nm{
			(Sg - S_h P_h g)(t)
		}_{L^2(\Omega)}  \lesssim
		\int_0^t \min\{s^{-1} h^2, 1+s^{\alpha-1} \} \, \mathrm{d}s
		\nm{g}_{L^\infty(0,T;L^2(\Omega))},
	\]
	so that the simple estimate
	\[
		\int_0^t \min\{s^{-1} h^2, 1+s^{\alpha-1} \} \, \mathrm{d}s
		\lesssim h^2 \ln(1/h)
	\]
  proves this lemma.
\end{proof}
\begin{rem}
  We note that \cite[Theorem 3.7]{Jin2013SIAM} implies
  \[
    \nm{
      \big(S\D_{0+}^\alpha v - S_h\D_{0+}^\alpha P_h v\big)(t)
    }_{L^2(\Omega)} \lesssim h^2 \ln(1/h) t^{-\alpha} \nm{v}_{L^2(\Omega)},
    \,\, 0 < t \leqslant T,
  \]
  for all $ v \in L^2(\Omega) $. We also notice that \cite[Theorem
  3.7]{Jin2015IMA} implies
  \[
    \nm{(Sg - S_h P_h g)(t)}_{L^2(\Omega)} \lesssim
    h^2 \ln(1/h)^2 \nm{g}_{L^\infty(0,t;L^2(\Omega))},
    \quad 0 \leqslant t \leqslant T,
  \]
  for all $ g \in L^\infty(0,T;L^2(\Omega)) $.
\end{rem}

%\begin{rem} % Todo: this remark needs to be commented.
  %How to remove the factor $ \ln h $? In fact, if $ g \in
  %{}_0H^{1/2}(0,T;L^2(\Omega)) $ then this factor can be removed indeed. However,
  %this is useless for our purpose.
%\end{rem}

%\begin{lem}
  %\label{lem:gronwall-2}
  %If $ g: [0,T] \to \mathbb R_{\geqslant 0} $ is continuous and satisfies that
  %\begin{equation}
    %\label{eq:19}
    %g(t) \leqslant \epsilon + A \D_{0+}^{-\alpha} g(t)
    %\quad\text{ for all }\, 0 < t \leqslant T,
  %\end{equation}
  %then
  %\begin{equation}
    %\label{eq:gronwall-2}
    %g(t) \leqslant C_{\alpha,A,T} \epsilon
    %\quad \text{ for all }\, 0 < t \leqslant T,
  %\end{equation}
  %where $ \epsilon $ and $ A $ are two positive constants.
%\end{lem}
%\begin{proof}
  %Letting $ \eta := \D_{0+}^{-\alpha} g $, by \cref{eq:19} we have
  %\[
    %\D_{0+}^\alpha \eta(t) \leqslant \epsilon + A \eta(t)
    %\quad \text{ for all } \, 0 < t \leqslant T.
  %\]
  %A simple calculation then yields
  %\[
    %\nm{\D_{0+}^\alpha \eta}_{L^p(0,t)} \leqslant
    %C_{\alpha,T} \big( \epsilon + A \nm{\eta}_{L^p(0,t)} \big)
    %\quad\text{ for all } \, 0 < t \leqslant T,
  %\]
  %where $ p:= 1/\alpha $. Hence, using \cite[Theorem 2.1]{Jin2018} gives
  %\[
    %\sup_{ 0 \leqslant t \leqslant T } \eta(t) \leqslant C_{\alpha,A,T} \epsilon,
  %\]
  %which, together with \cref{eq:19}, proves \cref{eq:gronwall-2}. This completes the
  %proof.
%\end{proof}

Let
\begin{equation}
	\label{eq:wt-u_h}
	\widetilde u_h := S_h\big( \D_{0+}^\alpha P_hu_0 + P_h f(u) \big).
\end{equation}
\begin{lem}
  \label{lem:u-wt-u_h}
  For any $ 0 < t \leqslant T $,
  \begin{equation}
    \label{eq:u-wt-u_h-1}
    \nm{(u-\widetilde u_h)(t)}_{L^2(\Omega)} \lesssim
    h^2( t^{-\alpha} + \ln(1/h) ) \nm{u_0}_{L^2(\Omega)}.
  \end{equation}
  Moreover,
  \begin{small}
  \begin{equation}
    \label{eq:u-wt-u_h-2}
    \nm{u-\widetilde u_h}_{{}_0H^{\alpha/2}(0,T;L^2(\Omega))} \lesssim
    \nm{u_0}_{L^2(\Omega)} \begin{cases}
    h^2 & \text{ if } 0 < \alpha < 1/3, \\
    \sqrt{\ln(1/h)} h^2 & \text{ if } \alpha = 1/3, \\
    h^{1/\alpha-1} & \text{ if } 1/3 < \alpha < 1,
  \end{cases}
\end{equation}
\end{small}
  and
  \begin{small}
  \begin{equation}
    \label{eq:u-wt-u_h-3}
    \begin{aligned}
      & \nm{u-\widetilde u_h}_{L^2(0,T;L^2(\Omega))} +
      h \nm{u-\widetilde u_h}_{L^2(0,T;\dot H^1(\Omega))} \\
      \lesssim{} &
      \nm{u_0}_{L^2(\Omega)}
      \begin{cases}
        h^2 & \text{ if }\, 0 < \alpha < 1/2, \\
        \sqrt{\ln(1/h)} \, h^2 & \text{ if }\, \alpha = 1/2, \\
        h^{1/\alpha} & \text{ if }\, 1/2 < \alpha < 1.
      \end{cases}
    \end{aligned}
  \end{equation}
	\end{small}
\end{lem}
\begin{proof}
  By \cref{lem:Lubich,thm:regu}, a simple calculation gives \cref{eq:u-wt-u_h-1}. By
  \cref{lem:coer,lem:regu-basic} and the regularity estimates in \cref{thm:regu}, a
  routine energy argument and a routine duality argument yield
  \cref{eq:u-wt-u_h-2,eq:u-wt-u_h-3}. This completes the proof.
\end{proof}

\medskip\noindent{\bf Proof of \cref{thm:space}.} Let us first prove \cref{eq:space}.
  By \cref{eq:u_h,eq:wt-u_h,eq:S_h-E_h} we have
  \[
    u_h(t) - \widetilde u_h(t) = \int_0^t
    E_h(t-s) \big( f(u_h(s)) - f(u(s)) \big) \, \mathrm{d}s,
  \]
  so that
  \begin{small}
  \begin{align*}
    & \nm{(u_h - \widetilde u_h)(t)}_{L^2(\Omega)} \\
    \lesssim{} & \int_0^t (t-s)^{\alpha-1}
    \nm{(u_h-u)(s)}_{L^2(\Omega)} \, \mathrm{d}s
    \quad \text{(by \cref{eq:E_h})} \\
    \lesssim{} & \int_0^t (t-s)^{\alpha-1} \Big(
      \nm{(u_h-\widetilde u_h)(s)}_{L^2(\Omega)} +
      \nm{(u - \widetilde u_h)(s)}_{L^2(\Omega)}
    \Big) \, \mathrm{d}s \\
    \lesssim{} &
    h^2 \ln(1/h) \nm{u_0}_{L^2(\Omega)} +
    \int_0^t(t-s)^{\alpha-1}
    \nm{(u_h-\widetilde u_h)(s)}_{L^2(\Omega)} \, \mathrm{d}s
    \quad \text{(by \cref{eq:u-wt-u_h-1}).}
  \end{align*}
  \end{small}
  Therefore, using \cite[Theorem 1.26]{Yagi2010} yields
  \[
    \nm{(u_h - \widetilde u_h)(t)}_{L^2(\Omega)} \lesssim
    h^2 \ln(1/h) \nm{u_0}_{L^2(\Omega)},
  \]
  which, together with \cref{eq:u-wt-u_h-1}, proves \cref{eq:space}.

  Then let us prove \cref{eq:u-u_h-energy}. Let $ \theta_h := u_h - \widetilde u_h $.
  For any $ 0 < t \leqslant T $, we have
  \[
    \dual{\D_{0+}^\alpha \theta_h, \theta_h}_{\Omega \times (0,t)} +
    \nm{\theta_h}_{L^2(0,t;\dot H^1(\Omega))}^2 =
    \dual{f(u_h) - f(u), \theta_h}_{\Omega \times (0,t)},
  \]
  so that \cref{lem:coer} implies
  \[
    \nm{\D_{0+}^{\alpha/2} \theta_h}_{L^2(0,t;L^2(\Omega))}^2 +
    \nm{\theta_h}_{L^2(0,t;\dot H^1(\Omega))}^2 \lesssim
    \dual{f(u_h) - f(u), \theta_h}_{\Omega \times (0,t)}.
  \]
  Since
  \begin{align*}
    & \dual{f(u_h) - f(u), \theta_h}_{\Omega \times (0,t)} \\
    \lesssim{} &
    \nm{u_h-u}_{L^2(0,t;L^2(\Omega))}
    \nm{\theta_h}_{L^2(0,t;L^2(\Omega))} \\
    \lesssim{} &
    \nm{\theta_h}_{L^2(0,t;L^2(\Omega))}^2 +
    \nm{u-\widetilde u_h}_{L^2(0,t;L^2(\Omega))}^2,
  \end{align*}
  it follows that, for any $ 0 < t \leqslant T $,
  \begin{align}
    & \nm{\D_{0+}^{\alpha/2}\theta_h}_{L^2(0,t;L^2(\Omega))}^2 +
    \nm{\theta_h}_{L^2(0,t;\dot H^1(\Omega))}^2 \notag \\
    \lesssim{} &
    \nm{u-\widetilde u_h}_{L^2(0,t;L^2(\Omega))}^2 +
    \nm{\theta_h}_{L^2(0,t;L^2(\Omega))}^2. \label{eq:killing}
  \end{align}
  Hence, using \cref{lem:gronwall} gives
  \[
    \nm{\D_{0+}^{\alpha/2}\theta_h}_{L^2(0,T;L^2(\Omega))}
    \lesssim \nm{u-\widetilde u_h}_{L^2(0,T;L^2(\Omega))},
  \]
  and so \cref{lem:regu-basic} implies
  \begin{equation}
    \label{eq:eve}
    \nm{\theta_h}_{{}_0H^{\alpha/2}(0,T;L^2(\Omega))}
    \lesssim \nm{u-\widetilde u_h}_{L^2(0,T;L^2(\Omega))}.
  \end{equation}
  Combining the above three estimates yields
  \begin{align*}
    \nm{\theta_h}_{{}_0H^{\alpha/2}(0,T;L^2(\Omega))} +
    \nm{\theta_h}_{L^2(0,T;\dot H^1(\Omega))} \lesssim
    \nm{u-\widetilde u_h}_{L^2(0,T;L^2(\Omega))}.
  \end{align*}
  It follows that
  \begin{align*}
    \nm{u-u_h}_{{}_0H^{\alpha/2}(0,T;L^2(\Omega))} & \lesssim
    \nm{u-\widetilde u_h}_{{}_0H^{\alpha/2}(0,T;L^2(\Omega))}, \\
    \nm{u-u_h}_{L^2(0,T;L^2(\Omega))} & \lesssim
    \nm{u-\widetilde u_h}_{L^2(0,T;L^2(\Omega))}, \\
    \nm{u-u_h}_{L^2(0,T;\dot H^1(\Omega))} & \lesssim
    \nm{u-\widetilde u_h}_{L^2(0,T;\dot H^1(\Omega))}.
  \end{align*}
  Therefore, using \cref{eq:u-wt-u_h-2,eq:u-wt-u_h-3} proves
  \cref{eq:u-u_h-frac,eq:u-u_h-energy}, respectively. This completes the proof.
\hfill\ensuremath{\blacksquare}

\section{Full discretization}
\label{sec:discretization}
%{{{ 2019-07-24  09:35
Assume that $ J \in \mathbb N_{>0} $ and $ \sigma \geqslant 1 $. Define
\begin{align*}
  t_j &:= (j/J)^\sigma T, \quad\,\,  0 \leqslant j \leqslant J, \\
  \tau_j &:= t_j - t_{j-1}, \quad 1 \leqslant j \leqslant J,
\end{align*}
and abbreviate $ \tau_J $ to $ \tau $ for convenience. Let
\[
  \mathcal W_{h,\tau} := \{
    V \in L^2(0,T; \mathcal V_h):\ V|_{(t_{j-1},t_j)}
    \text{ is constant valued for each } 1 \leqslant j \leqslant J
  \}.
\]
The full discretization of problem \cref{eq:model} reads as follows: seek $ U \in
\mathcal W_{h,\tau} $ such that
\begin{equation}
  \label{eq:numer_sol}
  \dual{\D_{0+}^\alpha(U-u_0), V}_{\Omega \times (0,T)} +
  \dual{\nabla U, \nabla V}_{\Omega \times (0,T)} =
  \dual{f(U), V}_{\Omega \times (0,T)}
\end{equation}
for all $ V \in \mathcal W_{h,\tau} $.

\begin{thm} % checked
  \label{thm:stab1}
  If
  \begin{equation}
    \label{eq:tau-cond}
    \tau < \Big( \frac1{L\Gamma(2-\alpha)} \Big)^{1/\alpha}
  \end{equation}
  then discretization \cref{eq:numer_sol} admits a unique solution $ U \in \mathcal
  W_{h,\tau} $.
\end{thm}
\begin{proof}
  Let us first prove that the solution $ U \in \mathcal W_{h,\tau} $ to
  discretization \cref{eq:numer_sol} is unique if it exists. To this end, we assume
  that $ \widetilde U $ is also a solution to problem \cref{eq:numer_sol}. Setting $
  e := U - \widetilde U $, from \cref{eq:numer_sol} we obtain that
  \[
    \dual{\D_{0+}^\alpha e, e}_{\Omega \times (0,t_1)} +
    \nm{e}_{L^2(0,t_1;\dot H^1(\Omega))}^2 =
    \dual{f(U) - f(\widetilde U), e}_{\Omega \times (0,t_1)}.
  \]
  Since
  \[
    \dual{f(U)-f(\widetilde U), e}_{\Omega \times (0,t_1)}
    \leqslant L \nm{e}_{L^2(0,t_1;L^2(\Omega))}^2,
  \]
  it follows that
  \[
    \dual{\D_{0+}^\alpha e, e}_{\Omega \times (0,t_1)}
    \leqslant L \nm{e}_{L^2(0,t_1;L^2(\Omega))}^2,
  \]
  which, together with the equality
  \[
    \dual{\D_{0+}^\alpha e, e}_{\Omega \times (0,t_1)} =
    \frac{t_1^{-\alpha}}{\Gamma(2-\alpha)}
    \nm{e}_{L^2(0,t_1;L^2(\Omega))}^2,
  \]
  implies
  \[
    \Big( \frac{t_1^{-\alpha}}{\Gamma(2-\alpha)} - L \Big)
    \nm{e}_{L^2(0,t_1;L^2(\Omega))}^2 \leqslant 0.
  \]
  Therefore, \cref{eq:tau-cond} indicates that $ e|_{(0,t_1)} = 0 $. Analogously, we
  can obtain sequentially that
  \[
    e|_{(t_1,t_2)} = 0, \, e|_{(t_2,t_3)} = 0, \, \ldots,
    e|_{(t_{J-1},t_J)} = 0.
  \]
  This proves the uniqueness of $ U $.

  Then let us prove that discretization \cref{eq:numer_sol} admits a solution $ U \in
  \mathcal W_{h,\tau} $. Observing that discretization \cref{eq:numer_sol} is a
  time-stepping scheme, we start by proving the existence of $ U $ on $ (0,t_1) $.
  This is equivalent to proving the existence of a zero of $ F: \mathcal V_h \to
  \mathcal V_h $, defined by that
  \begin{align*}
    \dual{FV, W}_\Omega =
    \dual{\D_{0+}^\alpha V, W}_{\Omega \times (0,t_1)} +
    \tau_1 \dual{\nabla V, \nabla W}_\Omega -
    \tau_1 \dual{f(V), W}_\Omega
  \end{align*}
  for all $ V, W \in \mathcal V_h $. For any $ V \in \mathcal V_h $, a simple
  calculation gives
  \begin{align*}
    \dual{FV,V}_\Omega & \geqslant
    \big( \tau_1^{1-\alpha} / \Gamma(2-\alpha) - L\tau_1 \big)
    \nm{V}_{L^2(\Omega)}^2 \geqslant C_{\alpha,\tau,L} \nm{V}_{L^2(\Omega)}^2.
  \end{align*}
  Applying the famous acute angle theorem (cf.~\cite[Chapter 9]{Evans2010}) then
  yields that $ F $ admits a zero indeed. This proves the existence of $ U $ on $
  (0,t_1) $. Similarly, we can prove sequentially that $ U $ exists on $ (t_1,t_2),
  (t_2,t_3), \ldots, (t_{J-1},t_J) $. Therefore, discretization \cref{eq:numer_sol}
  admits a solution indeed. This completes the proof.
\end{proof}

\begin{lem}
  \label{lem:gronwall-discr}
  Assume that $ \epsilon $ and $ A $ are two positive constants. If $ V \in \mathcal
  W_{h,\tau} $ satisfies that
  \begin{equation}
    \label{eq:831}
    \nm{\D_{0+}^{\alpha/2} V}_{L^2(0,t_j;L^2(\Omega))}^2
    \leqslant \epsilon + A \nm{V}_{L^2(0,t_j;L^2(\Omega))}^2
    \quad\text{ for all } 1 \leqslant j \leqslant J,
  \end{equation}
  then there exists a positive constant $ \tau^* $ depending only on $ \alpha $, $ A
  $ and $ T $ such that if $ \tau < \tau^* $ then
  \begin{equation}
    \label{eq:gronwall-discr}
    \nm{\D_{0+}^{\alpha/2}V}_{L^2(0,t_j;L^2(\Omega))} \leqslant
    \sqrt{\epsilon} \, C_{\alpha,A,T} \exp\big(
      j\sigma TJ^{-1} (\tau^*-\tau)^{-1}/2
    \big)
  \end{equation}
  for each $ 1 \leqslant j \leqslant J $.
\end{lem}
\begin{proof}
  Proceeding as in the proof of \cref{lem:gronwall} yields that there exists $ 1/2 <
  \gamma < 1 $ depending only on $ \alpha $ such that
  \begin{equation}
    \label{eq:shit20}
    \nm{\D_{0+}^\gamma g}_{L^2(0,t_j;L^2(\Omega))}^2 \leqslant
    C_{\alpha,A,T} \big(\epsilon + \nm{g}_{L^2(0,t_j;L^2(\Omega))}^2 \big)
  \end{equation}
  for each $ 1 \leqslant j \leqslant J $, where $ g := \D_{0+}^{-\gamma+\alpha/2} V $.
  Since
  \begin{align*}
    & \nm{g}_{L^\infty(0,t_j;L^2(\Omega))} =
    \nm{\D_{0+}^{-\gamma} \D_{0+}^\gamma g}_{L^\infty(0,t_j;L^2(\Omega))} \\
    ={} & \sup_{0 < t \leqslant t_j}
    \Nm{
      \frac1{\Gamma(\gamma)} \int_0^t (t-s)^{\gamma-1}
      \D_{0+}^\gamma g(s) \, \mathrm{d}s
    }_{L^2(\Omega)} \\
    \leqslant{} &
    C_\alpha \sup_{0 < t \leqslant t_j}
    \int_0^t (t-s)^{\gamma-1}
    \nm{\D_{0+}^\gamma g(s)}_{L^2(\Omega)} \, \mathrm{d}s \\
    \leqslant{} &
    C_{\alpha,T} \nm{\D_{0+}^\gamma g}_{L^2(0,t_j;L^2(\Omega))},
  \end{align*}
  it follows that
  \begin{equation}
    \label{eq:shit30}
    \nm{g}_{L^\infty(0,t_j;L^2(\Omega))}^2 \leqslant
    C_{\alpha,A,T} \big( \epsilon + \nm{g}_{L^2(0,t_j;L^2(\Omega))}^2 \big),
    \quad 1 \leqslant j \leqslant J.
  \end{equation}
  Letting $ \tau^* := 1/C_{\alpha,A,T} $ and assuming that $ \tau < \tau^* $, by
  \cref{eq:shit30} we obtain
  \begin{align*}
    \tau_j^{-1} (G_j - G_{j-1}) \leqslant
    \big( \epsilon + G_j \big) / \tau^*,
    \quad 1 \leqslant j \leqslant J,
  \end{align*}
  where $ G_j := \nm{g}_{L^2(0,t_j;L^2(\Omega))}^2 $. A straightforward computation
  then yields
  \begin{align*}
    G_j & \leqslant \epsilon/\tau^*
    \sum_{k=1}^j \tau_k \left(
      \prod_{m=k}^j (1-\tau_m/\tau^*)^{-1}
    \right) \\
    & < \epsilon/\tau^*
    \sum_{k=1}^j \tau (1-\tau/\tau^*)^{-(j-k+1)} \\
    & = \epsilon \tau/\tau^*
    \frac{(1-\tau/\tau^*)^{-j}-1}{\tau/\tau^*} \\
    & < \epsilon (1-\tau/\tau^*)^{-j}, \quad
    1 \leqslant j \leqslant J.
    %& < \epsilon \exp(j \tau/\tau^*(1-\tau/\tau^*)^{-1})
  \end{align*}
  From the estimate
  \[
    (1-\tau/\tau^*)^{-j} < \exp\big(j\tau/(\tau^*-\tau)\big)
    < \exp\big(j\sigma T J^{-1}(\tau^*-\tau)^{-1}\big),
  \]
  it follows that
  \[
    G_j < \epsilon \exp\big(j\sigma T J^{-1}(\tau^*-\tau)^{-1}\big),
    \quad 1 \leqslant j \leqslant J.
  \]
  This implies, by \cref{eq:shit20}, that
  \begin{align*}
    \nm{\D_{0+}^\gamma g}_{L^2(0,t_j;L^2(\Omega))}^2
    \leqslant{} &
    \epsilon C_{\alpha,A,T} \exp\big(
      j\sigma T J^{-1} (\tau^*-\tau)^{-1}
    \big), \quad 1 \leqslant j \leqslant J.
  \end{align*}
  Therefore, \cref{eq:gronwall-discr} follows from the fact
  \[
    \D_{0+}^{\alpha/2} V = \D_{0+}^\gamma g,
  \]
  and this lemma is thus proved.
\end{proof}

%}}}

In the rest of this paper, we assume that
\[
  \tau < \min\big\{ (L\Gamma(2-\alpha))^{-1/\alpha}, \tau^*/2 \big\},
  % \tau^*/2 is just a simple choice, for the sake of logic.
\]
where $ \tau^* $ is defined in \cref{lem:gronwall-discr}. We also assume $ J
\geqslant 2 $ for convenience.
Define
\begin{align}
  \eta_1(\alpha,\sigma,J)  &:=
  \begin{cases}
    J^{-\sigma/2} \sqrt{
      \frac{1-J^{\sigma+\alpha-2}}{
        2-\alpha-\sigma
      }
    }
    & \text{if } \sigma < 2-\alpha, \\
    J^{-(1-\alpha/2)} \sqrt{\ln J}
    & \text{if } \sigma = 2-\alpha, \\
    J^{-\sigma/2}  \sqrt{
      \frac{J^{\sigma+\alpha-2}-1}{
        \sigma + \alpha - 2
      }
    }
    & \text{if } 2-\alpha < \sigma < 3-\alpha,
  \end{cases} \label{eq:eta1} \\
  \eta_2(\alpha,\sigma,J) &:=
  \begin{cases}
    J^{-\sigma/2} \sqrt{
      \frac{1-J^{\sigma-2}}{ 2-\sigma }
    }
    & \text{if } \sigma < 2, \\
    J^{-1} \sqrt{\ln J} & \text{if } \sigma = 2, \\
    J^{-\sigma/2} \sqrt{
      \frac{J^{\sigma-2}-1}{ \sigma - 2 }
    } & \text{if } 2  <  \sigma < 3.
  \end{cases} \label{eq:eta2}
\end{align}
%\end{small}
The main results of this section are the following two error estimates.
\begin{thm}
	\label{thm:uh-U}
  It holds that
	\begin{equation}
		\label{eq:uh-U-1}
		\begin{aligned}
			& \nm{u_h-U}_{{}_0H^{\alpha/2}(0,T;L^2(\Omega))} +
			\nm{u_h-U}_{L^2(0,T;\dot H^1(\Omega))} \\
			& \quad{} + J^{\alpha/2} \nm{u_h-U}_{L^2(0,T;L^2(\Omega))} \\
			\lesssim{} &
			J^{-(1-\alpha)/2} \sqrt{\ln J} \, \nm{u_0}_{L^2(\Omega)}
		\end{aligned}
	\end{equation}
	and
	\begin{equation}
		\label{eq:uh-U-2}
		\begin{aligned}
			& \nm{u_h - U}_{{}_0H^{\alpha/2}(0,T;L^2(\Omega))} +
			\nm{u_h - U}_{L^2(0,T;\dot H^1(\Omega))} \\
			& \quad {} + J^{\alpha/2} \nm{u_h - U}_{L^2(0,T;L^2(\Omega))} \\
			\lesssim{} &
			\big(
				\eta_1(\alpha,\sigma,J) + \eta_2(\alpha,\sigma,J)
			\big) \nm{P_hu_0}_{\dot H^1(\Omega)}.
		\end{aligned}
	\end{equation}
\end{thm}
\begin{rem}
  By the techniques to prove \cref{eq:uh-U-2}, we can also obtain that
  \[
    \nm{u_h - U}_{{}_0H^{\alpha/2}(0,T;L^2(\Omega))} \lesssim
    \ln(1/h) J^{-(1-\alpha/2)} \nm{u_0}_{L^2(\Omega)},
  \]
  if $ \sigma $ is large enough and $ \mathcal K_h $ is quasi-uniform.
\end{rem}
\begin{rem}
  A simple modification of discretization \cref{eq:numer_sol} seeks $ \mathcal U \in
  \mathcal W_{h,\tau} $ such that
  \begin{small}
  \[
    \dual{\D_{0+}^\alpha(\mathcal U-u_0), v_h}_{\Omega \times (t_j,t_{j+1})} +
    \dual{\nabla \mathcal U, \nabla v_h}_{\Omega \times (t_j,t_{j+1})} =
    \dual{f(\mathcal U_j), v_h}_{\Omega \times (t_j,t_{j+1})}
  \]
  \end{small}
  for all $ v_h \in \mathcal V_h $ and $ 0 \leqslant j < J $, where $ \mathcal U_0 :=
  P_h u_0 $ and $ \mathcal U_j := \lim_{t \to {t_j-}} \mathcal U(t) $ for each $ 1
  \leqslant j < J $. The above discretization costs significantly less computation
  than that of discretization \cref{eq:numer_sol}. Interestingly, following the proof
  of \cref{eq:uh-U-1} we can easily obtain the error estimate
  \begin{small}
  \begin{equation*}
    \begin{aligned}
      & \nm{u_h - \mathcal U}_{{}_0H^{\alpha/2}(0,T;L^2(\Omega))} +
      \nm{u_h - \mathcal U}_{L^2(0,T;\dot H^1(\Omega))} \\
      & \quad {} +
      J^{\alpha/2} \nm{u_h - \mathcal U}_{L^2(0,T;L^2(\Omega))} \\
      \lesssim{} &
      J^{-(1-\alpha)/2} \sqrt{\ln J}
      \nm{u_0}_{L^2(\Omega)}.
    \end{aligned}
  \end{equation*}
  \end{small}
\end{rem}

The rest of this section is devoted to proving \cref{thm:uh-U}. Define $ \widetilde U
\in \mathcal W_{h,\tau} $ by that
\begin{equation}
  \label{eq:wt-U}
  \dual{\D_{0+}^\alpha(\widetilde U - u_0), V}_{\Omega \times (0,T)} +
  \dual{\nabla \widetilde U, \nabla V}_{\Omega \times (0,T)} =
  \dual{f(u_h), V}_{\Omega \times (0,T)}
\end{equation}
for all $ V \in \mathcal W_{h,\tau} $.
\begin{lem}
  \label{lem:uh-U}
  It holds that
  \begin{align}
    \nm{u_h - U}_{L^2(0,T;L^2(\Omega))} & \lesssim
    \nm{u_h - \widetilde U}_{L^2(0,T;L^2(\Omega))},
    \label{eq:uh-U-L2} \\
    \nm{u_h - U}_{L^2(0,T;\dot H^1(\Omega))} & \lesssim
    \nm{u_h - \widetilde U}_{L^2(0,T;\dot H^1(\Omega))},
    \label{eq:uh-U-H1} \\
    \nm{u_h - U}_{{}_0H^{\alpha/2}(0,T;L^2(\Omega))} & \lesssim
    \nm{u_h - \widetilde U}_{{}_0H^{\alpha/2}(0,T;L^2(\Omega))}.
    \label{eq:uh-U-frac}
  \end{align}
\end{lem}
\begin{proof}
  Let $ \theta := U - \widetilde U $. For any $ 1 \leqslant j \leqslant J $, from
  \cref{eq:numer_sol,eq:wt-U} we obtain
  \begin{align*}
    & \dual{\D_{0+}^\alpha \theta, \theta}_{\Omega \times (0,t_j)} +
    \nm{\theta}_{L^2(0,t_j;\dot H^1(\Omega))}^2 \\
    ={} &
    \dual{f(U) - f(u), \theta}_{\Omega \times (0,t_j)} \\
    \lesssim{} &
    \nm{u-U}_{L^2(0,t_j;L^2(\Omega))} \nm{\theta}_{L^2(0,t_j;L^2(\Omega))} \\
    \lesssim{} &
    \nm{u-\widetilde U}_{L^2(0,T;L^2(\Omega))}^2 +
    \nm{\theta}_{L^2(0,t_j;L^2(\Omega))}^2.
  \end{align*}
  From \cref{lem:coer} it follows that
  \begin{align*}
    & \nm{\D_{0+}^{\alpha/2}\theta}_{L^2(0,t_j;L^2(\Omega))}^2 +
    \nm{\theta}_{L^2(0,t_j;\dot H^1(\Omega))}^2 \\
    \lesssim{} &
    \nm{u-\widetilde U}_{L^2(0,T;L^2(\Omega))}^2 +
    \nm{\theta}_{L^2(0,t_j;L^2(\Omega))}^2
  \end{align*}
  for all $ 1 \leqslant j \leqslant J $, and hence \cref{lem:gronwall-discr} implies
  \begin{align*}
    \nm{\D_{0+}^{\alpha/2}\theta}_{L^2(0,T;L^2(\Omega))} \lesssim
    \nm{u-\widetilde U}_{L^2(0,T;L^2(\Omega))}.
  \end{align*}
  In addition,
  \[
    \nm{\theta}_{L^2(0,T;L^2(\Omega))} =
    \nm{
      \D_{0+}^{-\alpha/2} \D_{0+}^{\alpha/2} \theta
    }_{L^2(0,T;L^2(\Omega))} \lesssim
    \nm{\D_{0+}^{\alpha/2}\theta}_{L^2(0,T;L^2(\Omega))}.
  \]
  Consequently, combining the above three estimates yields
  \begin{align*}
    & \nm{\D_{0+}^{\alpha/2}\theta}_{L^2(0,T;L^2(\Omega))} +
    \nm{\theta}_{L^2(0,T;\dot H^1(\Omega))} +
    \nm{\theta}_{L^2(0,T;L^2(\Omega))} \\
    \lesssim{} &
    \nm{u_h-\widetilde U}_{L^2(0,T;L^2(\Omega))}.
  \end{align*}
  Therefore, by \cref{lem:regu-basic} and the estimate
  \begin{small}
  \[
    \nm{u_h-\widetilde U}_{L^2(0,T;L^2(\Omega))}
    \lesssim \min\big\{
      \nm{u_h - \widetilde U}_{{}_0H^{\alpha/2}(0,T;L^2(\Omega))},
      \nm{u_h - \widetilde U}_{L^2(0,T;\dot H^1(\Omega))}
    \big\},
  \]
  \end{small}
  we readily obtain \cref{eq:uh-U-L2,eq:uh-U-H1,eq:uh-U-frac}. This completes the
  proof.
\end{proof}

%\begin{lem}
  %%\label{lem:wt-U}
  %%If $ u_0 \in L^2(\Omega) $, then
  %%\begin{equation}
    %%\label{eq:wt-U-L2}
    %%\begin{aligned}
      %%& \nm{u_h - \widetilde U}_{{}_0H^{\alpha/2}(0,T;L^2(\Omega))} +
      %%\nm{u_h - \widetilde U}_{L^2(0,T;\dot H^1(\Omega))} \\
      %%& \quad {} + \tau^{-\alpha/2}
      %%\nm{u_h - \widetilde U}_{L^2(0,T;L^2(\Omega))} \\
      %%\lesssim{} &
      %%\tau^{(1-\alpha)/2} \sqrt{\ln(1/\tau)}
      %%\nm{u_0}_{L^2(\Omega)},
    %%\end{aligned}
  %%\end{equation}
  %%and
  %\begin{small}
  %\begin{equation}
    %\label{eq:wt-U-h1}
    %\begin{aligned}
      %& \nm{u_h - \widetilde U}_{{}_0H^{\alpha/2}(0,T;L^2(\Omega))} +
      %\nm{u_h - \widetilde U}_{L^2(0,T;\dot H^1(\Omega))} \\
      %& \quad {} +
      %J^{\alpha/2} \nm{u_h - \widetilde U}_{L^2(0,T;L^2(\Omega))} \\
      %\lesssim{} &
      %\big( \eta_1(\alpha,\sigma,J) + \eta_2(\alpha,\sigma,J) \big)
      %\nm{P_h u_0}_{\dot H^1(\Omega)}.
    %\end{aligned}
  %\end{equation}
  %\end{small}
%\end{lem}

\medskip\noindent{\bf Proof of \cref{thm:uh-U}.} Let $ P_\tau $ be defined by
  \cref{eq:P_tau}. By \cref{eq:u_h,eq:wt-U}, a standard energy argument gives
  \begin{small}
  \begin{align*}
    & \nm{u_h \!-\! \widetilde U}_{{}_0H^{\alpha/2}(0,T;L^2(\Omega))}
    \lesssim \nm{(I\!-\!P_\tau)u_h}_{{}_0H^{\alpha/2}(0,T;L^2(\Omega))}, \\
    & \nm{u_h \!-\!\widetilde U}_{L^2(0,T;\dot H^1(\Omega))}
    \lesssim \nm{(I\!-\!P_\tau)u_h}_{{}_0H^{\alpha/2}(0,T;L^2(\Omega))} +
    \nm{(I\!-\!P_\tau)u_h}_{L^2(0,T;\dot H^1(\Omega))},
  \end{align*}
  \end{small}
  and then a duality argument yields
  \begin{small}
  \begin{align*}
    & \nm{u_h - \widetilde U}_{L^2(0,T;L^2(\Omega))} \\
    \lesssim{} &
    J^{-\alpha/2} \big(
      \nm{u_h - \widetilde U}_{{}_0H^{\alpha/2}(0,T;L^2(\Omega))} +
      \nm{u_h - \widetilde U}_{L^2(0,T;\dot H^1(\Omega))}
    \big) \\
    \lesssim{} &
    J^{-\alpha/2} \big(
      \nm{(I-P_\tau)u_h}_{{}_0H^{\alpha/2}(0,T;L^2(\Omega))} +
      \nm{(I-P_\tau)u_h}_{L^2(0,T;\dot H^1(\Omega))}
    \big).
  \end{align*}
  \end{small}
  From \cref{lem:lxy} it follows that
  \begin{small}
  \begin{align*}
    & \nm{u_h - \widetilde U}_{{}_0H^{\alpha/2}(0,T;L^2(\Omega))} +
    \nm{u_h - \widetilde U}_{L^2(0,T;\dot H^1(\Omega))} \\
    & \quad {} +
    J^{\alpha/2} \nm{u_h - \widetilde U}_{L^2(0,T;L^2(\Omega))} \\
    \lesssim{} &
    \big( \eta_1(\alpha,\sigma,J) + \eta_2(\alpha,\sigma,J) \big)
    \nm{P_h u_0}_{\dot H^1(\Omega)},
  \end{align*}
  \end{small}
and hence by \cref{lem:uh-U} we obtain \cref{eq:uh-U-2}. Since \cref{eq:uh-U-1} can
  be proved analogously, this completes the proof.
\hfill\ensuremath{\blacksquare}

\section{Numerical results}
\label{sec:numer}
This section performs three numerical experiments in one dimensional space to verify
the theoretical results. Throughout this section, $ \Omega := (0,1) $, $ T := 1 $, $
f(s) := \sqrt{1+s^2} $ for all $ s \in \mathbb R $, and the spatial triangulation $
\mathcal K_h $ is uniform. Define
\begin{align*}
  \mathcal E_0 &:= \lim_{t \to {T-}} \nm{(U-U^*)(t)}_{\dot H^1(\Omega))}, \\
  \mathcal E_1 &:= \sqrt{
    \dual{\D_{0+}^\alpha(U-U^*), U-U^*}_{\Omega \times (0,T)}
  }, \\
  \mathcal E_2 &:= \nm{U-U^*}_{L^2(0,T;\dot H^1(\Omega))}, \\
  \mathcal E_3 &:= \nm{U-U^*}_{L^2(0,T;L^2(\Omega))},
\end{align*}
where $ U^* $, a reference solution, is the numerical solution with $ h = 2^{-11} $,
$ J = 2^{-16} $ and $ \sigma = 2.2 $. Additionally, the nonlinear systems arising in
the following numerical experiments are solved by the famous Newton's method, and the
stopping criterion is that the $ l^2 $-norm of the residual is less than $
1\mathrm{e-}\!13 $.

\medskip\noindent{\bf Experiment 1.} This experiment is to verify
\cref{thm:space} in the case that
\[
  u_0(x) := x^{-0.49}, \quad 0 < x < 1.
\]
By \cref{thm:space} we have the following predictions:
\begin{itemize}
\item $ \mathcal E_0 $ is close to $
O(h^2) $;

\item $ \mathcal E_1 $ is close to $ O(h^2) $ for $ \alpha \in \{0.2, 1/3\} $ and
close to $ O(h^{0.25}) $ for $ \alpha=0.8 $;

\item $ \mathcal E_2 $ is close to $ O(h) $
for $ \alpha \in \{0.1,0.5\} $ and close to $ O(h^{0.25}) $ for $ \alpha=0.8 $;

\item $
\mathcal E_3 $ is close to $ O(h^2) $ for $ \alpha \in \{0.1,0.5\} $ and close to $
O(h^{1.25}) $ for $ \alpha=0.8 $.
\end{itemize}
These predictions are verified by
\cref{tab:space-E0,tab:space-E1,tab:space-E23}.

\noindent
\begin{minipage}{\textwidth}
  \begin{minipage}[t]{0.49\textwidth}
    \begin{table}[H]
      \footnotesize \setlength{\tabcolsep}{0.5pt}
      \caption{$ J = 2^{-16} $, $ \sigma = 2.2 $}
      \label{tab:space-E0}
      \begin{tabular}{ccccccc}
        \toprule &
        \multicolumn{2}{c}{$\alpha=0.1$} &
        \multicolumn{2}{c}{$\alpha=0.5$} &
        \multicolumn{2}{c}{$\alpha=0.8$} \\
        \cmidrule(r){2-3} \cmidrule(r){4-5} \cmidrule(r){6-7}
        $h$ & $\mathcal E_0$ & Order & $\mathcal E_0$ & Order & $\mathcal E_0$ & Order \\
        $2^{-3}$ & 3.78e-3 & --   & 2.94e-3 & --   & 2.11e-3 & --   \\
        $2^{-4}$ & 9.86e-4 & 1.94 & 7.55e-4 & 1.96 & 5.33e-4 & 1.98 \\
        $2^{-5}$ & 2.56e-4 & 1.95 & 1.94e-4 & 1.96 & 1.34e-4 & 1.99 \\
        $2^{-6}$ & 6.61e-5 & 1.95 & 4.94e-5 & 1.97 & 3.38e-5 & 1.99 \\
        \bottomrule
      \end{tabular}
    \end{table}
  \end{minipage}
  \begin{minipage}[t]{0.5\textwidth}
    \begin{table}[H]
      \footnotesize \setlength{\tabcolsep}{0.5pt}
      \caption{$ J = 2^{-16} $, $ \sigma = 2.2 $}
      \label{tab:space-E1}
      \begin{tabular}{ccccccc}
        \toprule
        & \multicolumn{2}{c}{$\alpha=0.2$}
        & \multicolumn{2}{c}{$\alpha=1/3$}
        & \multicolumn{2}{c}{$\alpha=0.8$} \\
        \cmidrule(r){2-3} \cmidrule(r){4-5} \cmidrule(r){6-7}
        $ h $ & $ \mathcal E_1 $ & Order & $ \mathcal E_1 $ & Order & $ \mathcal E_1 $ &
        Order \\
        $ 2^{-3} $ & 4.41e-3 & --   & 5.86e-3 & --   & 4.37e-1 & --   \\
        $ 2^{-4} $ & 1.16e-3 & 1.93 & 1.64e-3 & 1.83 & 3.58e-1 & 0.28 \\
        $ 2^{-5} $ & 3.03e-4 & 1.94 & 4.55e-4 & 1.85 & 2.93e-1 & 0.29 \\
        $ 2^{-6} $ & 7.86e-5 & 1.94 & 1.24e-4 & 1.87 & 2.37e-1 & 0.31 \\
        \bottomrule
      \end{tabular}
    \end{table}
  \end{minipage}
	\begin{table}[H]
		\footnotesize \setlength{\tabcolsep}{1.0pt}
		\caption{$ J = 2^{-16} $, $ \sigma = 2.2 $}
		\label{tab:space-E23}
		\begin{tabular}{ccccccccccccc}
			\toprule &
			\multicolumn{4}{c}{$\alpha=0.1$} &
			\multicolumn{4}{c}{$\alpha=0.5$} &
			\multicolumn{4}{c}{$\alpha=0.8$} \\
			\cmidrule(r){2-5}
			\cmidrule(r){6-9}
			\cmidrule(r){10-13}
			$h$      & $\mathcal E_2$ & Order & $\mathcal E_3$ & Order & $\mathcal E_2$ & Order & $\mathcal E_3$ & Order & $\mathcal E_2$ & Order & $\mathcal E_3$ & Order \\
			$2^{-3}$ & 1.07e-1        & --    & 4.04e-3        & --    & 1.49e-1        & --    & 5.42e-3        & --    & 5.40e-1        & --    & 1.42e-2        & --    \\
			$2^{-4}$ & 5.57e-2        & 0.94  & 1.06e-3        & 1.93  & 8.26e-2        & 0.85  & 1.51e-3        & 1.84  & 4.40e-1        & 0.29  & 5.82e-3        & 1.29  \\
			$2^{-5}$ & 2.89e-2        & 0.95  & 2.75e-4        & 1.94  & 4.54e-2        & 0.86  & 4.18e-4        & 1.85  & 3.59e-1        & 0.29  & 2.41e-3        & 1.27  \\
			$2^{-6}$ & 1.49e-2        & 0.95  & 7.13e-5        & 1.95  & 2.47e-2        & 0.88  & 1.15e-4        & 1.87  & 2.92e-1        & 0.30  & 9.99e-4        & 1.27  \\
			\bottomrule
		\end{tabular}
	\end{table}
\end{minipage}

\medskip\noindent{\bf Experiment 2.} This experiment is to verify
error estimate \cref{eq:uh-U-1} in the case that
\[
  u_0(x) := x^{-0.49}, \quad 0 < x < 1.
\]
The numerical results displayed in \cref{tab:Ex2} illustrate that $ \mathcal E_1 $, $
\mathcal E_2 $ and $ \mathcal E_3 $ are close to $ O(J^{-(1-\alpha)/2}) $, $
O(J^{-(1-\alpha)/2}) $ and $ O(J^{-1/2}) $, respectively, which agrees well with
\cref{eq:uh-U-1}.
\begin{table}[H]
  \footnotesize
  \setlength{\tabcolsep}{6pt}
  \caption{$ h = 2^{-11} $, $ \sigma=1 $}
  \label{tab:Ex2}
  \begin{tabular}{cccccccc}
    \toprule
		& $J$ & $\mathcal E_1$ & Order & $\mathcal E_2$ & Order & $\mathcal E_3$ & Order \\
    \midrule
    \multirow{4}{*}{$\alpha=0.2$}
    & $2^9$    & 6.61e-3 & --   & 9.52e-3 & --   & 2.71e-3 & --   \\
    & $2^{10}$ & 5.24e-3 & 0.34 & 7.18e-3 & 0.41 & 2.01e-3 & 0.43 \\
    & $2^{11}$ & 4.12e-3 & 0.35 & 5.39e-3 & 0.41 & 1.48e-3 & 0.44 \\
    & $2^{12}$ & 3.21e-3 & 0.36 & 4.02e-3 & 0.42 & 1.08e-3 & 0.45 \\
    \midrule
    \multirow{4}{*}{$\alpha=0.5$}
    & $2^5$    & 1.42e-1 & --   & 1.45e-1 & --   & 3.22e-2 & --   \\
    & $2^6$    & 1.18e-1 & 0.27 & 1.17e-1 & 0.30 & 2.28e-2 & 0.50 \\
    & $2^7$    & 9.66e-2 & 0.29 & 9.60e-2 & 0.29 & 1.57e-2 & 0.54 \\
    & $2^8$    & 7.88e-2 & 0.29 & 7.97e-2 & 0.27 & 1.06e-2 & 0.56 \\
    \midrule
    \multirow{4}{*}{$\alpha=0.8$}
    & $2^5$    & 6.03e-1 & --   & 5.76e-1 & --   & 5.01e-2 & --   \\
    & $2^6$    & 5.54e-1 & 0.12 & 5.34e-1 & 0.11 & 3.42e-2 & 0.55 \\
    & $2^7$    & 5.12e-1 & 0.12 & 4.94e-1 & 0.11 & 2.37e-2 & 0.53 \\
    & $2^8$    & 4.73e-1 & 0.11 & 4.57e-1 & 0.11 & 1.66e-2 & 0.51 \\
    \bottomrule
  \end{tabular}
\end{table}

\medskip\noindent{\bf Experiment 3.} This experiment is to verify
 estimate \cref{eq:uh-U-2} in the case that
\[
  u_0(x) := x^{0.51}(1-x), \quad 0 < x < 1.
\]
Note that $ u_0 \in \dot H^{1.01-\epsilon}(\Omega) $ for all $ \epsilon > 0 $. We
summarize the corresponding numerical results as follows.
\begin{itemize}
  \item Estimate \cref{eq:uh-U-2} implies that $ \mathcal E_1 $ is close to $
    O(J^{-1/2}) $ for $ \sigma=1 $ and close to $ O(J^{-(1-\alpha/2)}) $ for $ \sigma
    \in \{2-\alpha,2\} $. For $ \alpha \in \{0.5,0.8\} $, the numerical results in
    \cref{tab:Ex3-E1} agree well with the theoretical results. For $ \alpha=0.2 $,
    the numerical results in \cref{tab:Ex3-E1} appear not to be in good agreement
    with the theoretical results. However, in our opinion this is caused by the
    limitation of the experiment: $ J $ can not be sufficiently large, or   it will
   leads to numerical instability.
  \item \cref{tab:Ex3-E2} illustrates that $ \mathcal E_2 $ is close to $
    O(J^{-\sigma/2}) $ for $ 1 \leqslant \sigma \leqslant 2$; however,  estimate
    \cref{eq:uh-U-2} only implies that $ \mathcal E_2 $ is close to $
    O(J^{-\sigma/2}) $ for $ 1 \leqslant \sigma \leqslant 2-\alpha $ and close to $
    O(J^{-(1-\alpha/2)}) $ for $ 2-\alpha < \sigma \leqslant 3 $. This phenomenon
    needs further analysis.
  \item \cref{tab:Ex3-E3} confirms the theoretical prediction that $ \mathcal E_3 $
    is close to $ O(J^{-(1+\alpha)/2}) $ for $ \sigma=1 $ and close to $ O(J^{-1}) $
    for $ \sigma \in \{2-\alpha,2\} $.
\end{itemize}

\begin{table}[H]
  \footnotesize
  \setlength{\tabcolsep}{2pt}
  \caption{$ h = 2^{-11} $}
  \label{tab:Ex3-E1}
  \begin{tabular}{cccccccc}
    \toprule & &
    \multicolumn{2}{c}{$\sigma=1$} &
    \multicolumn{2}{c}{$\sigma=2-\alpha$} &
    \multicolumn{2}{c}{$\sigma=2$} \\
    \cmidrule(r){3-4}
    \cmidrule(r){5-6}
    \cmidrule(r){7-8}
    & $J$ & $\mathcal E_1$ & Order & $\mathcal E_1$ & Order & $\mathcal E_1$ & Order \\
    \midrule
    \multirow{4}{*}{$\alpha=0.2$}
    & $2^{11}$ & 5.54e-4 & --   & 6.45e-5 & --   & 4.10e-5 & --   \\
    & $2^{12}$ & 4.31e-4 & 0.36 & 3.84e-5 & 0.75 & 2.33e-5 & 0.81 \\
    & $2^{13}$ & 3.33e-4 & 0.37 & 2.25e-5 & 0.77 & 1.31e-5 & 0.83 \\
    & $2^{14}$ & 2.55e-4 & 0.38 & 1.31e-5 & 0.78 & 7.39e-6 & 0.83 \\
    \midrule
    \multirow{4}{*}{$\alpha=0.5$}
    & $2^9$    & 6.35e-3 & --   & 2.02e-3 & --   & 1.01e-3 & --   \\
    & $2^{10}$ & 4.64e-3 & 0.45 & 1.28e-3 & 0.66 & 6.06e-4 & 0.73 \\
    & $2^{11}$ & 3.36e-3 & 0.47 & 8.07e-4 & 0.67 & 3.64e-4 & 0.74 \\
    & $2^{12}$ & 2.41e-3 & 0.48 & 5.04e-4 & 0.68 & 2.18e-4 & 0.74 \\
    \midrule
    \multirow{4}{*}{$\alpha=0.8$}
    & $2^9$    & 1.32e-2 & --   & 8.88e-3 & --   & 4.16e-3 & --   \\
    & $2^{10}$ & 9.53e-3 & 0.47 & 6.16e-3 & 0.53 & 2.76e-3 & 0.59 \\
    & $2^{11}$ & 6.86e-3 & 0.47 & 4.25e-3 & 0.53 & 1.83e-3 & 0.59 \\
    & $2^{12}$ & 4.91e-3 & 0.48 & 2.92e-3 & 0.54 & 1.22e-3 & 0.59 \\
    \bottomrule
  \end{tabular}
\end{table}
\begin{table}[H]
  \footnotesize
  \setlength{\tabcolsep}{2pt}
  \caption{$ h = 2^{-11} $}
  \label{tab:Ex3-E2}
  \begin{tabular}{cccccccc}
    \toprule & &
    \multicolumn{2}{c}{$\sigma=1$} &
    \multicolumn{2}{c}{$ \sigma=2-\alpha $} &
    \multicolumn{2}{c}{$\sigma=2$} \\
    \cmidrule(r){3-4}
    \cmidrule(r){5-6}
    \cmidrule(r){7-8}
    & $J$ & $\mathcal E_2$ & Order & $\mathcal E_2$ & Order & $\mathcal E_2$ & Order \\
    \midrule
    \multirow{4}{*}{$\alpha=0.2$}
    & $2^{10}$ & 8.90e-4 & --   & 9.29e-5 & --   & 5.97e-5 & --   \\
    & $2^{11}$ & 6.59e-4 & 0.43 & 5.18e-5 & 0.84 & 3.18e-5 & 0.91 \\
    & $2^{12}$ & 4.84e-4 & 0.44 & 2.86e-5 & 0.86 & 1.68e-5 & 0.92 \\
    & $2^{13}$ & 3.53e-4 & 0.45 & 1.57e-5 & 0.87 & 8.79e-6 & 0.93 \\
    \midrule
    \multirow{4}{*}{$\alpha=0.5$}
    & $2^8$    & 5.43e-3 & --   & 1.53e-3 & --   & 6.19e-4 & --   \\
    & $2^9$    & 3.79e-3 & 0.52 & 9.05e-4 & 0.76 & 3.25e-4 & 0.93 \\
    & $2^{10}$ & 2.65e-3 & 0.52 & 5.34e-4 & 0.76 & 1.70e-4 & 0.94 \\
    & $2^{11}$ & 1.85e-3 & 0.52 & 3.14e-4 & 0.76 & 8.82e-5 & 0.94 \\
    \midrule
    \multirow{4}{*}{$\alpha=0.8$}
    & $2^8$    & 8.75e-3 & --   & 5.16e-3 & --   & 1.12e-3 & --   \\
    & $2^9$    & 6.08e-3 & 0.52 & 3.35e-3 & 0.62 & 5.83e-4 & 0.94 \\
    & $2^{10}$ & 4.24e-3 & 0.52 & 2.18e-3 & 0.62 & 3.02e-4 & 0.95 \\
    & $2^{11}$ & 2.95e-3 & 0.52 & 1.42e-3 & 0.62 & 1.56e-4 & 0.95 \\
    \bottomrule
  \end{tabular}
\end{table}
\begin{table}[H]
  \footnotesize
  \setlength{\tabcolsep}{2pt}
  \caption{$ h = 2^{-11} $}
  \label{tab:Ex3-E3}
  \begin{tabular}{cccccccc}
    \toprule & &
    \multicolumn{2}{c}{$\sigma=1$} &
    \multicolumn{2}{c}{$\sigma=2-\alpha$} &
    \multicolumn{2}{c}{$\sigma=2$} \\
    \cmidrule(r){3-4}
    \cmidrule(r){5-6}
    \cmidrule(r){7-8}
    & $J$ & $\mathcal E_3$ & Order & $\mathcal E_3$ & Order & $\mathcal E_3$ & Order \\
    \midrule
    \multirow{4}{*}{$\alpha=0.2$}
    & $2^{11}$ & 1.99e-4 & --   & 1.46e-5 & --   & 9.01e-6 & --   \\
    & $2^{12}$ & 1.45e-4 & 0.45 & 7.83e-6 & 0.90 & 4.65e-6 & 0.95 \\
    & $2^{13}$ & 1.05e-4 & 0.47 & 4.16e-6 & 0.91 & 2.39e-6 & 0.96 \\
    & $2^{14}$ & 7.55e-5 & 0.48 & 2.19e-6 & 0.93 & 1.23e-6 & 0.96 \\
    \midrule
    \multirow{4}{*}{$\alpha=0.5$}
    & $2^8$    & 1.30e-3 & --   & 2.87e-4 & --   & 1.35e-4 & --   \\
    & $2^9$    & 8.16e-4 & 0.67 & 1.49e-4 & 0.95 & 6.74e-5 & 1.00 \\
    & $2^{10}$ & 5.04e-4 & 0.69 & 7.64e-5 & 0.96 & 3.37e-5 & 1.00 \\
    & $2^{11}$ & 3.07e-4 & 0.71 & 3.90e-5 & 0.97 & 1.68e-5 & 1.00 \\
    \midrule
    \multirow{4}{*}{$\alpha=0.8$}
    & $2^8$    & 1.08e-3 & --   & 5.48e-4 & --   & 1.95e-4 & --   \\
    & $2^9$    & 5.98e-4 & 0.86 & 2.81e-4 & 0.96 & 9.65e-5 & 1.02 \\
    & $2^{10}$ & 3.27e-4 & 0.87 & 1.43e-4 & 0.97 & 4.78e-5 & 1.02 \\
    & $2^{11}$ & 1.77e-4 & 0.88 & 7.25e-5 & 0.98 & 2.37e-5 & 1.01 \\
    \bottomrule
  \end{tabular}
\end{table}

\appendix
\section{Regularity of an initial value problem}
For any $ \beta >0 $, define the Mittag-Leffler function $ E_{\alpha,\beta}(z) $ by
\[
  E_{\alpha,\beta}(z) := \sum_{i=0}^\infty
  \frac{z^i}{\Gamma(i\alpha + \beta)}, \quad z \in \mathbb C,
\]
and this function admits the following growth estimate \cite{Podlubny1998}:
\begin{equation}
  \label{eq:ml_grow}
  \snm{E_{\alpha,\beta}(-t)} \leqslant
  \frac{C_{\alpha,\beta}}{1+t}
  \quad \forall t > 0.
\end{equation}
For any $ v \in L^2(\Omega) $, a routine calculation gives that \cite{Jin2013SIAM}
\begin{equation}
  \label{eq:SDu0}
  (S\D_{0+}^\alpha v)(t) = \sum_{k=0}^\infty
  E_{\alpha,1}(-\lambda_kt^\alpha) \dual{v,\phi_k}_\Omega \phi_k,
  \quad 0 < t \leqslant T.
\end{equation}
By the above two equations, a straightforward calculation gives the following lemma.
\begin{lem}
  \label{lem:regu-u0-growth}
  If $ v \in L^2(\Omega) $, then
  \[
    S\D_{0+}^\alpha v \in C([0,T];L^2(\Omega))
    \text{ with } (S\D_{0+}^\alpha v)(0) = v
  \]
  and
  \[
    \Dual{
      \D_{0+}^\alpha ( S\D_{0+}^\alpha v - v ), w
    }_{\Omega \times (0,T)} +
    \dual{\nabla S\D_{0+}^\alpha v,\nabla w}_{\Omega \times (0,T)} = 0
  \]
  for all $ w \in {}^0H^{\alpha/2}(0,T;L^2(\Omega)) \cap L^2(0,T;\dot H^1(\Omega)) $.
  If $ v \in \dot H^{2\delta}(\Omega) $ with $ 0 \leqslant \delta \leqslant 1 $,
  then
  \begin{small}
  \begin{equation}
    \label{eq:u0-L2-growth}
    \nm{(S\D_{0+}^\alpha v)(t)}_{L^2(\Omega)} +
    t^{1-\alpha\delta} \nm{(S\D_{0+}^\alpha v)'(t)}_{L^2(\Omega)} \\
    \leqslant C_\alpha \nm{v}_{\dot H^{2\delta}(\Omega)},
    \quad t > 0.
  \end{equation}
  \end{small}
  If $ v \in \dot H^{2\delta}(\Omega) $ with $ 1/2 \leqslant \delta \leqslant 1 $,
  then
  \begin{equation}
    \label{eq:u0-h1-growth}
    t^{1-\alpha(\delta-1/2)} \nm{(S\D_{0+}^\alpha v)'(t)}_{\dot H^1(\Omega)}
    \leqslant C_\alpha \nm{v}_{\dot H^{2\delta}(\Omega)},
    \quad t > 0.
  \end{equation}
\end{lem}

\begin{lem}
  \label{lem:regu-u0}
  Assume that $ v \in L^2(\Omega) $. For any $ 0 < \epsilon < 1/2 $,
  \begin{equation}
    \label{eq:regu-u0-11}
    \nm{S\D_{0+}^\alpha v}_{{}_0H^{1/2-\epsilon}(0,T;L^2(\Omega))}
    \leqslant \frac{C_{\alpha,T,\Omega}}{\sqrt{\epsilon(1-2\epsilon)}}
    \nm{v}_{L^2(\Omega)}.
  \end{equation}
  If $ 0 < \alpha < 1/3 $, then
  \begin{equation}
    \label{eq:regu-u0-1}
    \nm{S\D_{0+}^\alpha v}_{
      {}_0H^{\alpha/2}(0,T;\dot H^2(\Omega))
    } \leqslant C_{\alpha,T,\Omega}
    \nm{v}_{L^2(\Omega)}.
  \end{equation}
  If $ \alpha=1/3 $, then, for any $ 0 < \epsilon < 1 $,
  \begin{equation}
    \label{eq:regu-u0-2}
    \nm{S\D_{0+}^\alpha v}_{
      {}_0H^{\alpha/2}(0,T;\dot H^{2-\epsilon}(\Omega))
    } \leqslant C_{\alpha,T,\Omega} \epsilon^{-1/2}
    \nm{v}_{L^2(\Omega)}.
  \end{equation}
  If $ 1/3 < \alpha < 1 $, then
  \begin{equation}
    \label{eq:regu-u0-3}
    \nm{S\D_{0+}^\alpha v}_{
      {}_0H^{\alpha/2}(0,T;\dot H^{1/\alpha-1}(\Omega))
    } \leqslant C_{\alpha,T,\Omega}
    \nm{v}_{L^2(\Omega)}.
  \end{equation}
  If $ 0 < \alpha < 1/2 $, then
  \begin{equation}
    \label{eq:regu-u0-4}
    \nm{S\D_{0+}^\alpha v}_{
      L^2(0,T;\dot H^2(\Omega))
    } \leqslant C_{\alpha,T,\Omega} \nm{v}_{L^2(\Omega)}.
  \end{equation}
  If $ \alpha=1/2 $, then, for any $ 0 < \epsilon < 1 $,
  \begin{equation}
    \label{eq:regu-u0-5}
    \nm{S\D_{0+}^\alpha v}_{
      L^2(0,T;\dot H^{2-\epsilon}(\Omega))
    } \leqslant C_{\alpha,T,\Omega}
    \epsilon^{-1/2} \nm{v}_{L^2(\Omega)}.
  \end{equation}
  If $ 1/2 < \alpha < 1 $, then
  \begin{equation}
    \label{eq:regu-u0-6}
    \nm{S\D_{0+}^\alpha v}_{
      L^2(0,T;\dot H^{1/\alpha}(\Omega))
    } \leqslant C_{\alpha,T,\Omega} \nm{v}_{L^2(\Omega)}.
  \end{equation}
\end{lem}
\begin{proof}
	A straightforward calculation gives
	\begin{align*}
		\D_{0+}^{\alpha/2} (S\D_{0+}^\alpha v)(t) = t^{-\alpha/2}
		\sum_{k=0}^\infty E_{\alpha,1-\alpha/2}(-\lambda_k t^\alpha)
		\dual{v,\phi_k}_\Omega \phi_k.
	\end{align*}
	If $ 0 < \alpha < 1/3 $, then by \cref{eq:ml_grow} we obtain
	\begin{align*}
		& \nm{\D_{0+}^{\alpha/2}(S\D_{0+}^\alpha v)}_{
      L^2(0,T;\dot H^2(\Omega))
    }^2 \\
		\leqslant{} &
		C_\alpha \nm{v}_{L^2(\Omega)}^2
		\sup_{k \in \mathbb N} \int_0^T
		\frac{\lambda_k^2 t^{-\alpha}}{(1+\lambda_k t^\alpha)^2} \, \mathrm{d}t \\
		\leqslant{} &
		C_{\alpha,T,\Omega} \nm{v}_{L^2(\Omega)}^2,
	\end{align*}
	so that using \cref{lem:regu-basic} and the fact
	\[
		S\D_{0+}^\alpha v = \D_{0+}^{-\alpha/2}
    \D_{0+}^{\alpha/2}(S\D_{0+}^\alpha v)
	\]
  yields \cref{eq:regu-u0-1}. Since the rest of this lemma can be proved
  analogously, this completes the proof.
\end{proof}

\section{Two interpolation error estimates}
For any $ v \in L^1(0,T;L^2(\Omega)) $, define $ P_\tau v \in L^1(0,T;L^2(\Omega)) $
by that
\begin{equation}
  \label{eq:P_tau}
  P_\tau v|_{(t_{j-1},t_j)} :=
  \tau_j^{-1} \int_{t_{j-1}}^{t_j} v(t) \, \mathrm{d}t
\end{equation}
for all $ 1 \leqslant j \leqslant J $.

\begin{lem}
  \label{lem:lxy}
  It holds that
  \begin{align}
    \nm{(I-P_\tau)u_h}_{{}_0H^{\alpha/2}(0,T;L^2(\Omega))}
    & \leqslant C_{\alpha,L,T,\Omega} \eta_1(\alpha,\sigma,J)
    \nm{P_hu_0}_{\dot H^1(\Omega)}, \label{eq:lxy-1} \\
    \nm{(I-P_\tau)u_h}_{L^2(0,T;\dot H^1(\Omega))}
    & \leqslant C_{\alpha,L,T,\Omega} \eta_2(\alpha,\sigma,J)
    \nm{P_hu_0}_{\dot H^1(\Omega)}. \label{eq:lxy-2}
  \end{align}
\end{lem}
\begin{proof}
  Let $ g := (I-P_\tau) u_h $. By \cite[Lemmas 12.4 and 16.3]{Tartar2007}, a
  simple calculation gives
  \begin{align*}
    \nm{g}_{{}_0H^{\alpha/2}(0,T;L^2(\Omega))}^2 \leqslant
    C_\alpha \nm{g}_{H^{\alpha/2}(0,T;L^2(\Omega))}^2
    \leqslant C_{\alpha,T} \big(
      \mathbb I_1 + \mathbb I_2 + \mathbb I_3 + \mathbb I_4
    \big), % Note: H^{\alpha/2} is defined by the $K$-method.
  \end{align*}
  where
  \begin{small}
  \begin{align*}
    \mathbb I_1 &:=\int_0^{t_1} \int_0^{t_1}
    \frac{\nm{g(s)-g(t)}_{L^2(\Omega)}^2}{\snm{s-t}^{1+\alpha}}
    \, \mathrm{d}s \, \mathrm{d}t, \\
    \mathbb I_2 &:= \sum_{j=2}^J \int_{t_{j-1}}^{t_j} \int_{t_{j-1}}^{t_j}
    \frac{\nm{g(s)-g(t)}_{L^2(\Omega)}^2}{\snm{s-t}^{1+\alpha}}
    \, \mathrm{d}s \, \mathrm{d}t, \\
    \mathbb I_3 &:= \int_0^{t_1} \nm{g(t)}_{L^2(\Omega)}^2
    \big(
      (t_1\!-\!t)^{-\alpha} \!+\! t^{-\alpha}
    \big) \, \mathrm{d}t, \\
    \mathbb I_4 &:= \sum_{j=2}^J
    \int_{t_{j-1}}^{t_j} \nm{g(t)}_{L^2(\Omega)}^2
    \big(
      (t_j\!-\!t)^{-\alpha} \!+\! (t\!-\!t_{j-1})^{-\alpha}
    \big) \, \mathrm{d}t.
  \end{align*}
  \end{small}
  By \cref{eq:uh'-h1}, a routine calculation gives that
  \begin{small}
  \begin{align*}
    \mathbb I_1 & =
    2\int_0^{t_1} \, \mathrm{d}t \int_t^{t_1}  \frac{
      \nm{g(s) - g(t)}_{L^2(\Omega)}^2
    }{\snm{s-t}^{1+\alpha}} \, \mathrm{d}s \\
    & \leqslant
    C_{\alpha,L,T,\Omega} \int_0^{t_1} \, \mathrm{d}t
    \int_t^{t_1} (s-t)^{-1-\alpha} \big( s^{\alpha/2} - t^{\alpha/2} \big)^2
    \, \mathrm{d}s \nm{P_hu_0}_{\dot H^1(\Omega)}^2 \\
    & =
    C_{\alpha,L,T,\Omega}
    \int_0^{t_1}  \, \mathrm{d}t
    \int_1^{t_1/t} (y - 1)^{-1-\alpha} (y^{\alpha/2}-1)^2 \, \mathrm{d}y
    \nm{P_hu_0}_{\dot H^1(\Omega)}^2 \\
    & \leqslant
    C_{\alpha,L,T,\Omega}
    \int_0^{t_1} 1 + \ln(t_1/t) \, \mathrm{d}t
    \nm{P_hu_0}_{\dot H^1(\Omega)}^2 \\
    & \leqslant
    C_{\alpha,L,T,\Omega}
    \tau_1 \nm{P_hu_0}_{\dot H^1(\Omega)}^2
  \end{align*}
  \end{small}
  and
  \begin{align*}
   \mathbb I_3 & \leqslant C_{\alpha,L,T,\Omega}
    \tau_1^\alpha \int_0^{t_1} \big(
      (t_1-t)^{-\alpha} + t^{-\alpha}
    \big) \, \mathrm{d}t \nm{P_hu_0}_{\dot H^1(\Omega)}^2 \\
                & \leqslant C_{\alpha,L,T,\Omega}
                \tau_1 \nm{P_hu_0}_{\dot H^1(\Omega)}^2.
  \end{align*}
  Also, by \cref{eq:uh'-h1}, a straightforward calculation yields that
  \begin{small}
  \begin{align*}
    \mathbb I_2
    & = 2 \sum_{j=2}^J
    \int_{t_{j-1}}^{t_j} \int_{t}^{t_j}
    \frac{\nm{g(s) - g(t)}_{L^2(\Omega)}^2}{
      \snm{s-t}^{1+\alpha}
    } \, \mathrm{d}s \, \mathrm{d}t \\
    & \leqslant
    C_{\alpha,L,T,\Omega} \sum_{j=2}^J (t_{j-1}^{\alpha-1} - t_j^{\alpha-1})
    \int_{t_{j-1}}^{t_j} \int_t^{t_j}
    (s-t)^{-\alpha} \, \mathrm{d}s \, \mathrm{d}t
    \nm{P_hu_0}_{\dot H^1(\Omega)}^2 \\
    & \leqslant  C_{\alpha,L,T,\Omega} \sum_{j=2}^J
    \tau_j^{2-\alpha}(t_{j-1}^{\alpha-1} - t_j^{\alpha-1})
    \nm{P_hu_0}_{\dot H^1(\Omega)}^2
  \end{align*}
  \end{small}
  and
  \begin{small}
  \begin{align*}
    \mathbb I_4
    & \leqslant
    C_{\alpha,L,T,\Omega} \sum_{j=2}^J (t_{j-1}^{\alpha-1} - t_j^{\alpha-1})
    \tau_j \int_{t_{j-1}}^{t_j}
    \big( (t_j-t)^{-\alpha} + (t-t_{j-1})^{-\alpha} \big) \, \mathrm{d}t
    \nm{P_hu_0}_{\dot H^1(\Omega)}^2 \\
    & \leqslant C_{\alpha,L,T,\Omega} \sum_{j=2}^J
    \tau_j^{2-\alpha}(t_{j-1}^{\alpha-1} - t_j^{\alpha-1})
    \nm{P_hu_0}_{\dot H^1(\Omega)}^2.
  \end{align*}
  \end{small}
  Since
  \begin{small}
    \begin{align*}
    & \sum_{j=2}^J \tau_j^{2-\alpha} \big(
      t_{j-1}^{\alpha-1} -
      t_j^{\alpha-1}
    \big) \\
      ={} &
      J^{-\sigma}
      \sum_{j=2}^J \big( j^\sigma - (j-1)^\sigma \big)^{2-\alpha}
      ((j-1)^{\sigma(\alpha-1)} - j^{\sigma(\alpha-1)}) \\
      <{} &
      \sigma^{3-\alpha} (1-\alpha)
      J^{-\sigma} \sum_{j=2}^J
      j^{(\sigma-1)(2-\alpha)} (j-1)^{\sigma(\alpha-1) - 1} \\
      \leqslant{} &
      \sigma^{3-\alpha} (1-\alpha)
      J^{-\sigma} 2^{1 + \sigma(1-\alpha)} \sum_{j=2}^J
      j^{(\sigma-1)(2-\alpha)+\sigma(\alpha-1) - 1} \\
      <{} & \sigma^{3-\alpha} (1-\alpha)
      2^{1 + \sigma(1-\alpha)} \eta_1(\alpha,\sigma,J)^2,
    \end{align*}
  \end{small}
  it follows that
  \begin{align*}
    \mathbb I_2 + \mathbb I_4 \leqslant C_{\alpha,L,T,\Omega}
    \eta_1(\alpha,\sigma,J)^2 \nm{P_hu_0}_{\dot H^1(\Omega)}^2.
  \end{align*}
  Finally, combining the above estimates of $ \mathbb I_1 $, $ \mathbb I_2 $, $
  \mathbb I_3 $ and $ \mathbb I_4 $ proves \cref{eq:lxy-1}. Since the proof of
  \cref{eq:lxy-2} is similar, it is omitted. This completes the proof.
  %Then let us prove \cref{eq:lxy-2}. For any $ t_{j-1} \leqslant t \leqslant t_j $
  %with $ 2 \leqslant j \leqslant J $, we have
  %\begin{align*}
    %\nm{(u-P_\tau u)(t)}_{\dot H^1(\Omega)}^2 \leqslant
    %C_{\alpha,L,T,\Omega} \tau_j(t_{j-1}^{-1} - t_j^{-1}) \nm{u_0}_{\dot H^1(\Omega)}^2,
  %\end{align*}
  %so that
  %\begin{align*}
    %& \nm{(I-P_\tau)u}_{L^2(t_1,T;\dot H^1(\Omega))}^2 \\
    %\leqslant{} &
    %C_{\alpha,L,T,\Omega} \sum_{j=2}^J \tau_j^2
    %\big( t_{j-1}^{-1} - t_j^{-1} \big).
  %\end{align*}
  %Since
  %\begin{small}
  %\begin{align*}
     %& \sum_{j=2}^J \tau_j^2 (t_{j-1}^{-1} - t_j^{-1}) \\
    %={} & J^{-\sigma} \sum_{j=2}^J
    %(j^\sigma - (j-1)^\sigma)^2 ((j-1)^{-\sigma} - j^{-\sigma}) \\
    %<{} & J^{-\sigma} \sigma^3 2^{\sigma+1} \sum_{j=2}^J
    %j^{2(\sigma-1) - \sigma - 1} \\
    %={} & \sigma^3 2^{\sigma+1}
    %\eta_2(\alpha,\sigma,J)^2,
  %\end{align*}
  %\end{small}
\end{proof}


\begin{thebibliography}{10}

	\bibitem{Cuesta2006}
	E.~Cuesta, C.~Lubich, and C.~Palencia.
	\newblock Convolution quadrature time discretization of fractional
	diffusion-wave equations.
	\newblock {\em Math. Comput.}, 75(254):673--696, 2006.

	\bibitem{Ervin2006}
	V.~Ervin and J.~Roop.
	\newblock Variational formulation for the stationary fractional advection
	dispersion equation.
	\newblock {\em Numer. Meth. Part. D. E.}, 22(3):558--576, 2006.

	\bibitem{Evans2010}
	L.~C. Evans.
	\newblock {\em Partial differential equations}.
	\newblock American Mathematical Society, 2 edition, 2010.

	\bibitem{Jin2015IMA}
	B.~Jin, R.~Lazarov, J.~Pasciak, and Z.~Zhou.
	\newblock Error analysis of semidiscrete finite element methods for
	inhomogeneous time-fractional diffusion.
	\newblock {\em IMA. J. Numer. Anal.}, 35(2):561--582, 2015.

	\bibitem{Jin2013SIAM}
	B.~Jin, R.~Lazarov, and Z.~Zhou.
	\newblock Error estimates for a semidiscrete finite element method for
	fractional order parabolic equations.
	\newblock {\em SIAM J.  Numer. Anal.}, 51(1):445--466, 2013.

	\bibitem{Jin2016}
	B.~Jin, R.~Lazarov, and Z.~Zhou.
	\newblock Two fully discrete schemes for fractional diffusion and
	diffusion-wave equations with nonsmooth data.
	\newblock {\em SIAM J. Sci Comput.}, 38(1):A146--A170, 2016.

	\bibitem{Jin2018}
	B.~Jin, B.~Li, and Z.~Zhou.
	\newblock Numerical analysis of nonlinear subdiffusion equations.
	\newblock {\em SIAM J. Numer. Anal.}, 56(1):1--23, 2018.

	\bibitem{Li2018-wave}
	B.~Li, H.~Luo, and X.~Xie.
	\newblock A time-spectral algorithm for fractional wave problems.
	\newblock {\em J. Sci. Comput.}, 77(2):1164--1184, 2018.

	\bibitem{Li2019SIAM}
	B.~Li, H.~Luo, and X.~Xie.
	\newblock Analysis of a time-stepping scheme for time fractional diffusion
	problems with nonsmooth data.
	\newblock {\em SIAM J. Numer. Anal.}, 57(2):779--798, 2019.

	\bibitem{Li-Wang-Xie2019wave}
	B.~Li, T.~Wang, and X.~Xie.
	\newblock Analysis of the L1 scheme for fractional wave equations with nonsmooth
	data.
	\newblock {\em submitted}, axXiv:1908.09145.

	\bibitem{Li-Wu-Zhang2019}
	D.~Li, C.~Wu, and Z.~Zhang.
	\newblock Linearized galerkin fems for nonlinear time fractional parabolic
	problems with non-smooth solutions in time direction.
	\newblock {\em J. Sci. Comput.}, 80(1):403--419, 2019.

	\bibitem{Li2009}
	X.~Li and C.~Xu.
	\newblock A space-time spectral method for the time fractional diffusion
	equation.
	\newblock {\em SIAM J. Numer. Anal.}, 47(3):2108--2131, 2009.
\bibitem{Liao2019}
H.~Liao, W.~McLean and J.~Zhang.
	\newblock A discrete Gr\"{o}nwall inequality with application to numerical schemes for subdiffusion problems
		\newblock {\em SIAM J. Numer. Anal.}, 57(1):218--237, 2019.
	\bibitem{Lin2007}
	Y.~Lin and C.~Xu.
	\newblock Finite difference/spectral approximations for the time-fractional
	diffusion equation.
	\newblock {\em J. Comput. Phys.}, 225(2):1533--1552, 2007.

	\bibitem{Lubich1986}
	C.~Lubich.
	\newblock Discretized fractional calculus.
	\newblock {\em SIAM J. Math. Anal.}, 17(3):704--719, 1986.

	\bibitem{Lubich1988}
	C.~Lubich.
	\newblock Convolution quadrature and discretized operational calculus.
	\newblock {\em Numer. Math.}, 52(4):129--145, 1988.

	\bibitem{Lubich1996}
	C.~Lubich, I.~Sloan, and V.~Thom\'ee.
	\newblock Nonsmooth data error estimates for approximations of an evolution
	equation with a positive-type memory term.
	\newblock {\em Math. Comput.}, 65(213):1--17, 1996.

	\bibitem{Luo2019}
	H.~Luo, B.~Li., and X.~Xie.
	\newblock Convergence analysis of a petrov--galerkin method for fractional wave
	problems with nonsmooth data.
	\newblock {\em J. Sci. Comput.}, 80(2):957--992, 2019.

	\bibitem{Mclean2009Convergence}
	W.~McLean and K.~Mustapha.
	\newblock Convergence analysis of a discontinuous galerkin method for a
	sub-diffusion equation.
	\newblock {\em Numer. Algor.}, 52(1):69--88, 2009.

	\bibitem{Mclean2015Time}
	W.~McLean and K.~Mustapha.
	\newblock Time-stepping error bounds for fractional diffusion problems with
	non-smooth initial data.
	\newblock {\em J. Comput. Phys.}, 293:201--217, 2015.

	\bibitem{Mustapha2009Discontinuous}
	K.~Mustapha and W.~McLean.
	\newblock Discontinuous Galerkin method for an evolution equation with a memory
	term of positive type.
	\newblock {\em Math.  Comput.}, 78(268):1975--1995, 2009.

	\bibitem{Mustapha2012Uniform}
	K.~Mustapha and W.~McLean.
	\newblock Uniform convergence for a discontinuous Galerkin, time-stepping
	method applied to a fractional diffusion equation.
	\newblock {\em IMA J.  Numer. Anal.}, 32(3):906--925(20), 2012.

	\bibitem{Mustapha2010IMA}
	K.~Mustapha and H.~Mustapha.
	\newblock A second-order accurate numerical method for a semilinear
	integro-differential equations with a weakly singular kernel.
	\newblock {\em IMA J.  Numer. Anal.}, 30:555--578, 2010.

	\bibitem{Podlubny1998}
	I.~Podlubny.
	\newblock {\em Fractional differential equations}.
	\newblock Academic Press, 1998.

	\bibitem{sun2006fully}
	Z.~Sun and X.~Wu.
	\newblock A fully discrete difference scheme for a diffusion-wave system.
	\newblock {\em Appl. Numer. Math.}, 56(2):193--209, 2006.

	\bibitem{Tartar2007}
	L.~Tartar.
	\newblock {\em An Introduction to Sobolev Spaces and Interpolation Spaces}.
	\newblock Springer, Berlin, 2007.

	\bibitem{Wang-Chen-Xiao2019}
	R.~Wang, D.~Chen, and T.~Xiao.
	\newblock Abstract fractional cauchy problems with almost sectorial operators.
	\newblock {\em J. Differ. Equations}, 252(1):202--235, 2012.

	\bibitem{Yagi2010}
  A.~Yagi.
	\newblock Abstract parabolic evolution equations and their applications.
	\newblock Springer, Berlin, 2010.

	\bibitem{Zayernouri2014Fractional}
	M.~Zayernouri and G.~E.~ Karniadakis.
	\newblock Fractional spectral collocation method.
	\newblock {\em SIAM J. Sci. Comput.}, 36(1):A40--A62, 2014.

	\bibitem{Zayernouri2014Exponentially}
	M.~Zayernouri and G.~E. Karniadakis.
	\newblock Exponentially accurate spectral and spectral element methods for
	fractional odes.
	\newblock {\em J. Comput. Phys.}, 257(2):460--480, 2014.

	\bibitem{Zayernouri2012Karniadakis}
	M.~Zayernouri and G.~E.~Karniadakis.
	\newblock  Discontinuous spectral element methods for time- and space-fractional advection equations.
	\newblock {\em SIAM J. Sci. Comput.}, 36(4):B684--B707, 2014.

	\bibitem{Zheng2015}
	M.~Zheng, F.~Liu, I.~Turner, and V.~Anh.
	\newblock A novel high order space-time method for the time fractional
	fokker-planck equation.
	\newblock {\em SIAM J. Sci. Comput.}, 37(2):A701--A724, 2015.
\end{thebibliography}
\end{document}